\documentclass[11pt]{article}

\usepackage[utf8]{inputenc}
\usepackage[T1]{fontenc}
\usepackage{lmodern} 
\usepackage{microtype}
\usepackage{amsmath,amssymb,amsthm,mathtools,stmaryrd}
\usepackage{enumitem}
\usepackage{hyperref}
\usepackage[titletoc,title]{appendix}
\usepackage[nameinlink,capitalise]{cleveref}
\usepackage{mathrsfs}
\usepackage{booktabs}
\usepackage{tabularx}
\usepackage{array}
\usepackage{float}
\usepackage{needspace}
\newcolumntype{Y}{>{\raggedright\arraybackslash}X}
\newcolumntype{P}[1]{>{\raggedright\arraybackslash}p{#1}}
\newcolumntype{L}[1]{>{\raggedright\arraybackslash}p{#1}}
\newcolumntype{L}{>{\raggedright\arraybackslash}X}
\newcolumntype{C}{>{\centering\arraybackslash}p{0.20\textwidth}}
\newcolumntype{M}{>{\raggedright\arraybackslash}p{0.28\textwidth}}

\DeclareMathOperator{\Fix}{Fix}

\theoremstyle{plain}
\newtheorem{theorem}{Theorem}[section]
\newtheorem{lemma}[theorem]{Lemma}
\newtheorem{corollary}[theorem]{Corollary}
\newtheorem{proposition}[theorem]{Proposition}
\newtheorem{construction}[theorem]{Construction}
\theoremstyle{remark}
\newtheorem{remark}[theorem]{Remark}
\theoremstyle{definition}
\newtheorem{definition}[theorem]{Definition}
\newtheorem{example}[theorem]{Example}
\newtheorem{claim}[theorem]{Claim}

\crefname{theorem}{theorem}{theorems}
\Crefname{theorem}{Theorem}{Theorems}
\crefname{lemma}{lemma}{lemmas}
\Crefname{lemma}{Lemma}{Lemmas}
\crefname{corollary}{corollary}{corollaries}
\Crefname{corollary}{Corollary}{Corollaries}
\crefname{remark}{remark}{remarks}
\Crefname{remark}{Remark}{Remarks}
\crefname{example}{example}{examples}
\Crefname{example}{Example}{Examples}
\crefname{definition}{definition}{definitions}
\Crefname{definition}{Definition}{Definitions}
\crefname{proposition}{proposition}{propositions}
\Crefname{proposition}{Proposition}{Propositions}
\crefname{claim}{claim}{claims}
\Crefname{claim}{Claim}{Claims}

\newcommand{\Boxsym}{\mathop{\Box_{\mathrm{sym}}}}

\newcommand{\HS}{\mathsf{HS}}
\newcommand{\IS}{\mathsf{IS}}
\newcommand{\Aut}{\mathrm{Aut}}
\newcommand{\Sym}{\mathrm{Sym}}
\newcommand{\Add}{\mathrm{Add}}
\newcommand{\forces}{\Vdash}

\newcommand{\ZF}{\mathsf{ZF}}
\newcommand{\ZFC}{\mathsf{ZFC}}

\newcommand{\Sel}{\mathrm{Sel}}
\newcommand{\Pblk}{\mathcal{P}}
\newcommand{\Qblk}{\mathcal{Q}}
\newcommand{\Rblk}{\mathcal{R}}
\newcommand{\forcesIS}{\Vdash^{\IS}}
\newcommand{\FSI}{finite symmetry-preserving iteration}
\newcommand{\FSIs}{finite symmetry-preserving iterations}
\newcommand{\symext}{symmetric extension}         

\providecommand{\Nessym}{\Box_{\mathrm{sym}}}      
\providecommand{\Possym}{\Diamond_{\mathrm{sym}}}  

\title{The Modal Logic of Finitely Symmetry-Preserving Iterated Extensions is Exactly S4}
\author{Frank Gilson}
\date{June 21, 2026}

\begin{document}
	\maketitle
	
	\begin{abstract}
		We determine the $\ZF$-provable modal logic of the modality $\Boxsym$, where
		$\Boxsym\varphi$ means ``$\varphi$ holds in every \emph{\FSI},'' i.e.,
		along finite, symmetry-preserving iterations of the symmetric method. We prove that the
		exact logic is \textbf{S4}. Soundness (axioms T and 4) follows from reflexivity and transitivity of
		the underlying accessibility relation. Exactness is obtained by (i) a non-amalgamation
		lemma showing that axiom $\mathbf{.2}$ fails for \FSIs \ (no common \FSI\ above the parent; see \Cref{lem:non-amalgamation}), and (ii) a
		$p$-morphism/finite-frame realization producing, within $\ZF$, models whose $\Boxsym$-theory
		matches any finite reflexive–transitive frame.
	\end{abstract}
	
	\section{Introduction}
	\paragraph{Standing background.}
	We work in a ground $V\models\ZFC$ as a metatheory for the forcing and symmetry constructions; generic filters are assumed to exist externally. The forcing $\Add(\omega,\omega)$ is homogeneous, and we use its homogeneity explicitly when passing from automorphism-invariance to the evaluation notion of finite support.
	
	\emph{Metatheoretic convention.} Completeness is established externally 
	in $\mathsf{ZFC}$: if $\alpha\notin\mathrm{S4}$, we (in the metatheory) 
	build a generic extension and take a symmetric \emph{ZF} submodel $N$, 
	which is a transitive model of $\mathsf{ZF}$ in $V$. The operator 
	$\Boxsym$ is interpreted externally: $N\models\neg\,\Boxsym\alpha$ means 
	that there exists in $V$ a finite \FSI{} $N'$ of $N$ (also a transitive 
	model of $\mathsf{ZF}$ in $V$) with $N'\models\neg\alpha$. Since $\Boxsym$ 
	quantifies over generics existing in $V$ but not in $N$, it is not 
	$\mathsf{ZF}$-definable inside $N$; all such quantification is metatheoretic. 
	All reasoning \emph{inside} the symmetric models is in $\mathsf{ZF}$.
	
	Symmetric extensions are the standard tool for producing models of $\ZF$ with various failures of choice. They sit strictly between ground $V$ and a generic extension $V[G]$ by modding out names with respect to a group of automorphisms and a normal filter of subgroups. This paper studies the modal operator
	\[
	\Boxsym\varphi \quad:\Longleftrightarrow\quad \text{``$\varphi$ holds in every \FSI''}.
	\]
	\paragraph{Possible vs.\ necessary.}
	We read $\Nessym\varphi$ as “$\varphi$ is \emph{necessary} under \FSIs{}”
	(i.e., true in every \FSI), and we read
	$\Possym\varphi$ as “$\varphi$ is \emph{possible} under \FSIs”
	(i.e., true in some \FSI). We treat $\Possym$ as a
	\emph{derived} operator by duality:
	\[
	\Possym\varphi \;:=\; \neg\,\Nessym\neg\varphi.
	\]
	All results below are stated for $\Nessym$; by duality they immediately transfer to $\Possym$.
	
	In particular, because the accessibility relation arising from finite,
	symmetry-preserving iterations is not directed (Section 6), the axiom
	\(\mathbf{.2}\) fails in this setting. We prove below that the ZF-provable
	valid principles of \(\Boxsym\) are exactly \textbf{S4}.
	
	\paragraph{Scope and relation to prior work.}
	Throughout, $\Boxsym$ quantifies over \emph{\FSIs{}} (“true in every \FSI”). By contrast, Block--L\"owe\footnote{Block--L\"owe quantify over \emph{all symmetric submodels of forcing extensions} $V[G]$. By Proposition~\ref{prop:full-filter}, taking the full normal filter (generated by $\{1\}$) yields the \emph{entire} forcing extension $V^P$ as a symmetric extension. Hence their class includes all forcing extensions and is directed under products: given $V^P$ and $V^Q$ there is a common extension $V^{P\times Q}$ (and more generally, see Grigorieff \cite{Grigorieff1975} on intermediate submodels in product extensions). Directedness validates axiom~$\mathbf{.2}$, so the resulting modal logic is $\mathrm{S4.2}$, as in the forcing case (cf.\ Hamkins--L\"owe \cite{HamkinsLowe2008}). Our setting restricts to \emph{finite, symmetry-preserving iterations} (Def.~\ref{def:SPI}) with specific normal filters; the resulting class is not directed (Lemma~\ref{lem:sibling-incompatibility}), so $\mathbf{.2}$ fails and the exact logic drops to $\mathrm{S4}$.}
	study a translation restoring directedness and validating $\mathbf{.2}$, hence \textbf{S4.2} \cite[Thm.~18]{BlockLoewe2015}. Our non-amalgamation (Lemma~\ref{lem:sibling-incompatibility})
	shows that directedness fails inside the finite, symmetry-preserving iteration regime we study.
	
	A recent note of Duncan \cite{Duncan2026} studies the modal logic of the 
	full symmetric extension multiverse (in the sense of Block--L\"owe), 
	proving that independent systems of choice-switches cannot be independent 
	from standard cardinal-property buttons. Duncan works with the unrestricted 
	class of all symmetric extensions of models of $\mathsf{ZFC}$, for which the 
	modal logic remains $\mathrm{S4.2}$. Our result is complementary: by 
	restricting to \emph{finite symmetry-preserving iterations} with the 
	productive filter discipline, the resulting class loses directedness 
	(Lemma~\ref{lem:sibling-incompatibility}), and the exact modal logic drops 
	to $\mathrm{S4}$.
	
	\paragraph{Historical context.}
	The symmetric method traces back to early work on forcing and failures of choice; see Feferman~\cite{Feferman1964}. 
	Intermediate submodels of generic extensions were analyzed by Grigorieff~\cite{Grigorieff1975}, which conceptually underlies using symmetric submodels as definable intermediates between $V$ and $V[G]$.
	
	\section{Modal and proof-theoretic preliminaries}
	
	We work with the standard propositional modal language over a base theory of $\ZF$; $\textbf{S4}$ denotes the normal modal logic with $\neg,\wedge,\to,\Box,\Diamond$. \textbf{Soundness, completeness, finite model property} (FMP), and $p$-morphism facts are used in the usual way; see \cite[\S2.3]{BlackburnDeRijkeVenema} or \cite[\S5.3]{ChagrovZakharyaschev}. Throughout, $\Add(\omega,\omega)$ is Cohen forcing. (For a quick-reference table of recurring symbols such as $\Add(\omega,\omega)$, $IS$, $\HS_{\mathcal F}$, $J_d$, and $G_{J_d}$, see Appendix~\ref{app:notation}, Table~\ref{tab:notation}.)
	
	\begin{definition}[Cohen forcing $\Add(\omega,\omega)$]\label[definition]{def:Add-omega-omega}
		Conditions are finite partial functions $p:\omega\times\omega\to 2$ ordered by reverse inclusion ($p\le q$ iff $p\supseteq q$).
		For $n\in\omega$, let $c_n$ denote the $n$th Cohen real added by the generic filter. Thus $\Add(\omega,\omega)$ adds a sequence
		$\langle c_n:n\in\omega\rangle$ of Cohen reals.
	\end{definition}
	
	\paragraph{Homogeneity.}
	We recall that $\Add(\omega,\omega)$ is homogeneous; see, e.g., \cite[Lemma~14.17]{Jech}.
	We use this to pass from automorphism-invariance to agreement of evaluations in \Cref{lem:finite-support} (see Lemmas~4.3–4.4 and Corollary~4.5).
	
	\begin{definition}[Finite support for evaluation]\label[definition]{def:finite-support}
	Work in $V$ and let $\mathbb P=\Add(\omega,\omega)$ with coordinates
	$\langle c_n:n\in\omega\rangle$. A $\mathbb P$-name $\dot x$ has \emph{finite support}
	if there exists a finite set $F\subseteq\omega$ such that for every pair of
	$V$-generics $G,H\subseteq\mathbb P$ with $G\!\restriction F=H\!\restriction F$
	we have $\dot x^G=\dot x^H$. Any such $F$ is called a (evaluation) \emph{support} for $\dot x$.
	
	Equivalently, there is a finite $F\subseteq\omega$ such that every finitary permutation
	$\pi$ of $\omega$ with $\operatorname{supp}(\pi)\cap F=\varnothing$ fixes $\dot x$
	(i.e., $\pi\cdot\dot x=\dot x$). In this case we also say that $\dot x$ is \emph{fixed off $F$}.
	\end{definition}
	
	\begin{remark}[Standard usage]
		This matches the classical ``$\mathrm{Fix}(E)$ support'' perspective in permutation models.
	\end{remark}
	
	\begin{remark}[Meaning of “ZF-provably valid”]\label{rem:ZFprovable}
		By “ZF-provably valid” we mean: ZF proves that the modal axiom holds under the
		$\Box_{\mathrm{sym}}$ semantics (i.e., in every finite, symmetry-preserving iteration).
		Equivalently, if $\alpha\notin\mathrm{S4}$ then (in ZFC) there exists a model of ZF
		refuting $\Box_{\mathrm{sym}}\alpha$.
	\end{remark}
	
	\begin{remark}[Support calculus]\label[remark]{rem:support-calculus}
		If a name $\dot x$ has finite support $F$ (in the evaluation sense of \Cref{def:finite-support}), then:
		\begin{enumerate}[leftmargin=2em]
			\item (\emph{Conjugation}) If $\pi$ is an automorphism decided in the stage below, then $\pi\cdot\dot x$ has support $\pi[F]$.
			\item (\emph{Finite intersections}) If $\dot x,\dot y$ have supports $F_x,F_y$, then any pairing/union formed from them has support contained in $F_x\cup F_y$, and any property depending on both is decided by $G\!\restriction(F_x\cup F_y)$.
			\item (\emph{Excellent supports}) In productive steps, excellent supports exist and are preserved under finite intersections (see \cite[§4]{Karagila2019}).
		\end{enumerate}
		This is the support bookkeeping used implicitly in §§3–5.
	\end{remark}
	
	\begin{remark}[Evaluation vs.\ symmetry support]\label[remark]{rem:eval-vs-sym}
		We distinguish the evaluation notion of finite support (agreement of generics on $F$ implies equal evaluations, Definition~\ref{def:finite-support}) from the automorphism–invariance notion (fixed by a tail stabilizer). Cohen homogeneity bridges the two (\Cref{lem:finite-support}); we use the evaluation phrasing in \S\ref{sec:partition-construction} when describing symmetric systems and finite iterations. Automorphism support is syntactic, while evaluation support is semantic.
	\end{remark}
	
	\begin{remark}[Terminology]\label[remark]{rem:terminology}
	When we speak of a \emph{finite symmetry-preserving iteration}, we mean the notion formalized in Definition~\ref{def:SPI}: a finite productive iteration (PI–1 through PI–4 in \cite{Karagila2019}) obeying our fixed block-partition discipline.
	
	We use \emph{\symext} for the one-step construction of a symmetric ZF submodel arising from a single symmetric system. 
	We use \emph{\FSI} (abbrev.\ FSI) for any finite iteration over $V$ composed of our productive symmetric steps (PI–1–PI–4). 
	Accordingly, phrases like “no common extension” always mean “no common \FSI{} above the parent” unless explicitly stated otherwise.
	\end{remark}
	
	\section{Symmetric systems and finite iterations}\label{sec:symmetric-iterations}
	\paragraph{Notation discipline.}
	Stage-indexed objects carry the stage as a subscript: $\mathbb P_\alpha$ denotes the $\alpha$th iterand, and ``$\mathbb P_\alpha$-name'' means a name built over the ground appropriate to stage $\alpha$. 
	``Ground name'' always refers to the earlier stage in a current factorization. We write $\HS_F$ for the hereditarily symmetric \emph{names} associated to a normal filter $F$; when needed we use $\HS_\alpha$ or $\HS_{F_\alpha}$ to indicate the stage. For a $V$-generic $H\subseteq\mathbb P$, we write $IS=\HS_F^{\,H}$ for the interpreted symmetric \emph{model}\footnote{A compact notation table is in Appendix~\ref{app:notation} (Table~\ref{tab:notation}).}.
	
	A \emph{symmetric system} is $(\mathbb{P},G,\mathcal{F})$ with $\mathbb P$ a notion of forcing, $G\leq \Aut(\mathbb P)$ a group of automorphisms, and $\mathcal F$ a normal filter of subgroups of $G$. The hereditarily symmetric class of names is $\HS_{\mathcal F}$. For any $V$-generic $H\subseteq\mathbb P$, its interpretation $IS=\HS_{\mathcal F}^{\,H}$ is a transitive model with $V\subseteq IS\subseteq V^{\mathbb P}$.
	
	\begin{proposition}[Forcing as a special case of symmetry]\label[proposition]{prop:full-filter}
		Let $(\Pblk,G,\mathcal F_{\mathrm{full}})$ be a symmetric system where
		$\mathcal F_{\mathrm{full}}$ is the \emph{full} normal filter of subgroups of $G$
		(i.e., the upward-closed, conjugation-closed family of \emph{all} subgroups of $G$,
		generated by the trivial subgroup $\{1\}$).
		Then every $\Pblk$-name is hereditarily symmetric and hence $IS^{H}_{\mathcal F_{\mathrm{full}}}(\Pblk)=V^\Pblk$ for any $V$-generic $H\subseteq \Pblk$.
		In particular, every forcing extension is (canonically) a symmetric extension.
	\end{proposition}
	
	\begin{proof}
		A name $\dot x$ is (hereditarily) symmetric iff its symmetry group
		$\mathrm{sym}(\dot x)\le G$ belongs to the normal filter. Since
		$\mathcal F_{\mathrm{full}}$ contains \emph{all} subgroups of $G$, we have
		$\mathrm{sym}(\dot x)\in\mathcal F_{\mathrm{full}}$ for every name $\dot x$, and this persists hereditarily.
		Thus $HS_{\mathcal F_{\mathrm{full}}}$ is the full class of $\Pblk$-names and the interpreted model equals $V^\Pblk$.
	\end{proof}
	
	\subsection{Composition: \FSIs{} collapse to one step}
	Karagila proves that \FSIs{} can be compressed to a single symmetric step over the ground. We recall the representation we use and cite the exact places where bracketing/mixing are applied (Lemmas 7.5–7.6) and where the finite collapse to the ground is performed (Theorems 7.8–7.9) in \cite{Karagila2019}.
	
	\begin{definition}[Symmetry-preserving iteration]\label{def:SPI}
		A finite sequence $\langle(\mathbb P_i,\mathcal G_i,\mathcal F_i): i<n\rangle$ of symmetric systems is
		\emph{symmetry-preserving} if for each $i<n-1$:
		\begin{enumerate}
			\item \textbf{Lift compatibility.} In $V^{\mathbb P_i}$, the next step is a symmetric system 
			$(\dot{\mathbb P}_{i+1},\dot{\mathcal G}_{i+1},\dot{\mathcal F}_{i+1})$, and the stage-$(i{+}1)$ automorphism group extends the \emph{lift} of the previous-stage action to the two-step iteration 
			$\mathbb P_i\!\ast\!\dot{\mathbb P}_{i+1}$. 
			Here the lift of $\pi\in\mathcal G_i$ is the automorphism $\pi^{\uparrow}$ given by
			\[
			\pi^{\uparrow}\!\cdot\!(p,\dot q)\ :=\ (\pi p,\,\pi\dot q).
			\]
			Thus $\mathcal G_{i+1}\supseteq \{\pi^{\uparrow}:\pi\in\mathcal G_i\}$.
			\item \textbf{Filter monotonicity.} The pushforward of $\mathcal F_i$ along the lift is contained in $\mathcal F_{i+1}$; equivalently, if $K\in\mathcal F_i$ then the pointwise stabilizer $\mathrm{Fix}^{\uparrow}(K)$ belongs to $\mathcal F_{i+1}$.
			\item \textbf{Productive hypotheses.} Both the current step $(\mathbb P_i,\mathcal G_i,\mathcal F_i)$ and the next step satisfy Karagila’s productive-iteration conditions (PI–1)–(PI–4) \cite{Karagila2019}.
		\end{enumerate}
		We say the iteration follows the \emph{fixed-filter discipline} if the block/partition data that generate $\mathcal F_{i+1}$ are fixed at stage $i$ and are never relaxed at later stages (only finitely many additional blocks are fixed at each step).
		
		We allow $n=0$; the \emph{null iteration} is the empty sequence, 
		with $\operatorname{HS}_\varnothing = V$, so $V$ is an \FSI{} of itself. 
		The conditions above are vacuously satisfied.
	\end{definition}
	
	\begin{remark}[How this is used later]
		Under the fixed-filter discipline, siblings at the same depth use opposite block-partitions while a branch preserves its partition label; this is exactly the setup exploited by the sibling non-amalgamation lemma (later Lemma~\ref{lem:sibling-incompatibility}) and in the §6 template construction.
	\end{remark}

	\subsection{Factoring through intermediate HS-stages}
	\paragraph{Notation.} For a fixed \FSI{}, we write $N_k$ for the $k$th intermediate symmetric extension over $V$. If an intermediate stage is $N=IS^{H}_{\mathcal{F}}(\mathbb P)$ (the interpretation of the HS-class via $H$), we write $\HS_{\mathcal F}(\mathbb P)$ for the HS-class of names and $IS^{H}_{\mathcal F}(\mathbb P)$ for its interpretation; in particular, $N=IS^{H}_{\mathcal F}(\mathbb P)$.
	
	\begin{lemma}[Equivariant bracketing and factoring]\label[lemma]{lem:factoring}
		Let $N_0\subseteq N_1$ be two successive symmetric stages, with the top step presented as $(\mathbb P_0,G_0,\mathcal F_0)$ and forcing $\dot{\mathbb P}_1$ over $\HS_{\mathcal F_0}$. If $\dot x$ is \emph{hereditarily} symmetric for $(\mathbb P_0,G_0,\mathcal F_0)*\dot{\mathbb P}_1$, then there is a hereditarily symmetric ground name $\dot\tau\in \HS_{\mathcal F_0}$ such that, for any sentence $\varphi$, we have
		\[
		V^{\mathbb P_0*\dot{\mathbb P}_1}\forces\varphi(\dot x)\quad\Longleftrightarrow\quad V^{\mathbb P_0}\forces \varphi(\dot\tau).
		\]
	\end{lemma}
	\begin{proof}
		Fix a presentation of the top step over $N_0$ as $(\mathbb P_0,G_0,\mathcal F_0)$, forcing $\dot{\mathbb P}_1$ over $\HS_{\mathcal F_0}$.
		Working in $V$, let $G_0\subseteq \mathbb P_0$ be generic. By \cite[Thm.~5.2]{Karagila2019}, in $V[G_0]$ there is a maximal antichain $D\subseteq \dot{\mathbb P}_1^{G_0}$ such that for each $q\in D$ there is an $N_0$-hereditarily symmetric name $\dot x_q\in \HS_{\mathcal F_0}$ with
		\[
		V^{\mathbb P_0*\dot{\mathbb P}_1}\forces \dot x = \sum_{q\in D} (\,\check q\,\wedge \dot x_q\,).
		\]
		Here the right-hand side is a $\mathbb P_0*\dot{\mathbb P}_1$-name (formed in $V$):
		by convention $(\check q\wedge \dot x_q)$ denotes the pair $\langle \dot x_q,\check q\rangle$.
		To pass to a \emph{ground} $\mathbb P_0$-name, for each $q\in D$ let $\dot d_q$ be the canonical
		$\mathbb P_0$-name for “$q\in\dot G_1$”, and form the mixed $\mathbb P_0$-name
		\[
		\sum_{q\in D}\dot d_q\cdot \dot x_q \;=\; \{\langle \dot x_q,\dot d_q\rangle : q\in D\}.
		\]
		Moreover, each $\dot x_q$ can be chosen $N_0$-hereditarily symmetric over $(\mathbb P_0,G_0,\mathcal F_0)$,
		i.e.\ $\dot x_q\in \HS_{\mathcal F_0}$.
		
		\smallskip\noindent\emph{Orbit-mixing recipe.}
		Partition $D$ into $G_0$-orbits. For each orbit $O\subseteq D$, fix a representative $r\in O$.
		For every $p\in O$ choose $\pi_p\in G_0$ with $\pi_pr=p$ and set $\dot x_p:=\pi_p\cdot \dot x_r$.
		Define
		\[
		\dot\tau_O:=\sum_{p\in O} \dot d_p\cdot \dot x_p
		\qquad\text{and}\qquad
		\dot\tau:=\sum_{O} \dot\tau_O,
		\]
		where $\dot d_p$ is the canonical $\mathbb P_0$-name for ``$p\in \dot G_1$''. Since any
		$\pi\in G_0$ permutes each orbit $O$, we have $\pi\dot\tau_O=\dot\tau_O$ and hence
		$\pi\dot\tau=\dot\tau$. Thus $G_0\le\mathrm{sym}(\dot\tau)$ and $\dot\tau\in \HS_{\mathcal F_0}$,
		while evaluation in any $G_0*G_1$ agrees with that of the original $\dot x$ by construction.
		
		Since each $\dot x_p$ is hereditarily symmetric and mixing respects subnames,
		$\dot\tau$ is hereditarily symmetric as a $\mathbb P_0$-name. Therefore
		$V^{\mathbb P_0}\forces \varphi(\dot\tau)$ iff
		$V^{\mathbb P_0*\dot{\mathbb P}_1}\forces \varphi(\dot x)$ for every sentence $\varphi$.
	\end{proof}
	
	\begin{lemma}[Synchronized seeds on a common ground antichain]\label{lem:synch-seeds}
		Work in the setting of Corollary~\ref{cor:intersection} after compressing the finite iteration to a single symmetric step over~$V$. 
		Apply Lemma~\ref{lem:factoring} along each factorization $V\to N_\Pblk\to M$ and $V\to N_\Qblk\to M$ to obtain presentations
		\[
		x_P=\sum_{p\in D_\Pblk} \dot d_p\cdot t^\Pblk_p\in HS_{\Pblk},\qquad
		x_Q=\sum_{q\in D_\Qblk} \dot d_q\cdot t^\Qblk_q\in HS_{\Qblk},
		\]
		where $D_\Pblk,D_\Qblk$ are ground maximal antichains of the top iterand and $\{\dot d_\bullet\}$ are the canonical indicator names.
		Fix a ground refinement $D\subseteq D_\Pblk\cap D_\Qblk$, and for each $p\in D$ let $p_\Pblk\in D_\Pblk$ and $p_\Qblk\in D_\Qblk$ be the unique conditions with $p\le p_\Pblk$ and $p\le p_\Qblk$.
		Then for every generic $G$ for the top forcing with $p\in G$ we have
		\[
		\bigl(t^\Pblk_{p_\Pblk}\bigr)^G=\bigl(t^\Qblk_{p_\Qblk}\bigr)^G.
		\]
		Moreover, by choosing orbit representatives compatibly \emph{before} refining to $D$, we may arrange that
		\[
		t^\Pblk_{p_\Pblk}=t^\Qblk_{p_\Qblk}\quad\text{as \emph{names} in $V$ for all $p\in D$}.
		\]
	\end{lemma}
	
	\begin{proof}
		By Lemma~\ref{lem:factoring}, if $p\in D\subseteq D_\Pblk\cap D_\Qblk$ and $G\ni p$ is generic, then $x_\Pblk^G=(t^\Pblk_{p_\Pblk})^G$ and $x_\Qblk^G=(t^\Qblk_{p_\Qblk})^G$ because the indicator $\dot d_p$ selects the unique seed above $p$. The two presentations arise from the same top step and evaluate to the same object in~$M$, hence $(t^\Pblk_{p_\Pblk})^G=(t^\Qblk_{p_\Qblk})^G$ for every $G\ni p$.
		Assume toward a contradiction that $t^{\mathcal P}_{p_{\mathcal P}}\neq t^{\mathcal Q}_{p_{\mathcal Q}}$ as ground names. 
		Then there is a ground formula $\varphi(x)$ such that, refining if necessary, we have conditions 
		$p'\le p_{\mathcal P}$ and $q'\le p_{\mathcal Q}$ with 
		$p'\Vdash\varphi\!\bigl(t^{\mathcal P}_{p_{\mathcal P}}\bigr)$ and 
		$q'\Vdash\neg\varphi\!\bigl(t^{\mathcal Q}_{p_{\mathcal Q}}\bigr)$. 
		Factor through the parent antichain $D$ (Lemma~\ref{lem:factoring}): pick $d\in D$ meeting both $p',q'$ and write the synchronized indicator for $d$ as $\dot d_{d}$. 
		By construction of synchronized seeds, both $t^{\mathcal P}_{p_{\mathcal P}}$ and $t^{\mathcal Q}_{p_{\mathcal Q}}$ are obtained by applying the same ground Borel functional $\Theta$ to the common parent trace decided by $d$, i.e.
		\[
		t^{\mathcal R}_{p_{\mathcal R}}\ =\ \Theta\bigl(\mathrm{tr}(d)\bigr)\qquad(\mathcal R\in\{\mathcal P,\mathcal Q\}).
		\]
		Hence they are \emph{equal} as ground names, contradicting the displayed forcing. 
		Equivalently (homogeneity view): using a finitary permutation that fixes the parent trace of $d$ pointwise, we can transport $p'$ to a condition compatible with $q'$ without changing the value of the synchronized seed, contradicting $\varphi$ vs.\ $\neg\varphi$. By the forcing theorem there exist a formula \(\varphi(u)\) and a condition \(r\le p\) in the top forcing with \(r\Vdash \varphi(t^{\Pblk}_{p_\Pblk})\) and \(r\Vdash \neg\varphi(t^{\Qblk}_{p_\Qblk})\). By density below \(p\) and maximality of the ground antichain \(D\), extend \(r\) to a \(V\)-generic \(G\) with \(p\in G\); then \((t^{\Pblk}_{p_\Pblk})^{G}\neq (t^{\Qblk}_{p_\Qblk})^{G}\), contradicting the fact established in the previous paragraph that \((t^{\Pblk}_{p_\Pblk})^{G}=(t^{\Qblk}_{p_\Qblk})^{G}\) for all \(G\ni p\). 
		Finally, choosing orbit representatives for the two actions compatibly prior to refining to $D$ and transporting representatives by the respective automorphisms yields identical seed assignments on~$D$.
	\end{proof}
	
	\begin{definition}[Combined symmetry]\label[definition]{def:combined-symmetry}
		Let \(H := \langle G_{\Pblk}, G_{\Qblk}\rangle\), i.e., the subgroup generated by \(G_{\Pblk} \cup G_{\Qblk}\), be the subgroup of the ambient automorphism group generated by $G_{\Pblk}$ and $G_{\Qblk}$. 
		It acts (in $V$) on the common ground antichain $D$ from Lemma~\ref{lem:synch-seeds}.
	\end{definition}
	
	\begin{remark}\label{rem:combined-vs-semidirect}
		Throughout \S\ref{sec:symmetric-iterations} we use the \emph{combined symmetry} $H:=\langle G_{\mathcal P},G_{\mathcal Q}\rangle$ (Def.~\ref{def:combined-symmetry})
		and the standard lift of automorphisms to products/iterations (see §3.1), together with one-shot $H$-orbit mixing over a fixed ground maximal antichain $D$ (synchronizing seeds and then mixing, as in Lemmas~3.2–3.5). 
		All arguments go through using $H$ and the lifted action.
	\end{remark}
	
	\begin{definition}[Orbit mixing along a group action]\label{def:orbit-mixing}
		Work in the setting of Lemma~\ref{lem:factoring}.
		Let $\Gamma$ act (in $V$) on a ground maximal antichain $D$ of the top forcing, and suppose we have
		$V$-assigned names $\{\dot x_p : p\in D\}$ with each $\dot x_p\in \HS_{\mathcal F_0}$.
		For every $\Gamma$-orbit $O\subseteq D$ fix a representative $r_O\in O$ and for each $p\in O$ pick $\pi_p\in\Gamma$ with $\pi_pr_O=p$.
		Let $\dot d_p$ be the canonical $\mathbb P_0$-name for “$p\in\dot G_1$”. Define the mixed $\Gamma$-invariant $\mathbb P_0$-name
		\[
		\dot\tau^\Gamma\ :=\ \sum_{O\in\mathrm{Orb}_\Gamma(D)}\ \sum_{p\in O}\ \dot d_p\cdot \bigl(\pi_p\cdot \dot x_{r_O}\bigr).
		\]
		Then for every $\gamma\in\Gamma$ we have $\gamma\dot\tau^\Gamma=\dot\tau^\Gamma$, so $\Gamma\le \mathrm{sym}(\dot\tau^\Gamma)$ and $\dot\tau^\Gamma\in \HS_{\mathcal F_0}$.
		Moreover, if $G_0*G_1$ is generic with $p^\ast\in D\cap G_1$, then
		\[
		(\dot\tau^\Gamma)^{G_0}\ =\ \bigl(\pi_{p^\ast}\cdot \dot x_{r_O}\bigr)^{G_0},
		\]
		where $O$ is the $\Gamma$-orbit of $p^\ast$; in particular this agrees with the two-stage evaluation of the original name from Lemma~\ref{lem:factoring}.
	\end{definition}
	
	\begin{lemma}[One-shot $H$-orbit mixing gives simultaneous invariance]\label{lem:H-mix}
		With $D$ and synchronized seeds $t_p$ as in Lemma~\ref{lem:synch-seeds}, define $\dot x'$ by orbit mixing
		\emph{(via Definition~\ref{def:orbit-mixing} with $\Gamma=H$)}:
		\[
		\dot\tau_O:=\sum_{p\in O} \dot d_p\cdot t_p
		\quad\text{and}\quad
		\dot x':=\sum_{O\in \mathrm{Orb}_H(D)} \dot\tau_O.
		\]
		Then every $\gamma\in H$ permutes the summands within each $O$, hence $\gamma\dot x'=\dot x'$. 
		Therefore $\dot x'\in HS_{\Pblk}\cap HS_Q$, and for any generic $G$ meeting $p\in D$ we have $(\dot x')^G = t_p^G$, which agrees with the evaluations of $x_P$ and $x_Q$.
	\end{lemma}
	
	\begin{proof}
		$H$ permutes each $H$-orbit $O$ of $D$, so $\gamma\dot\tau_O=\dot\tau_O$ for all $\gamma\in H$; thus $\gamma\dot x'=\dot x'$. 
		Fixation by $G_{\Pblk}$ (resp.\ $G_{\Qblk}$) implies $\dot x'\in HS_{\Pblk}$ (resp.\ $\dot x'\in HS_Q$); hereditariness follows from the closure of HS under mixing and subnames.
		Evaluation is by the indicator mechanism on $D$ as in Lemma~\ref{lem:factoring}.
	\end{proof}
	
	\begin{lemma}[Signal invariance]\label[lemma]{lem:signal-invariance}
		The sentences $S_{\mathcal P}$ and $S_{\mathcal Q}$ are invariant under $G_{\mathcal P}$ and $G_{\mathcal Q}$, respectively. \noindent\textit{Definability note.} The family \(A_\Rblk\) is first-order definable from ground parameters (the fixed partition \(\Rblk\) and the canonical coordinate scheme); hence every \(\pi\in G_\Rblk\) preserves \(A_\Rblk\) as a parameter and acts only by permuting elements within each block.
	\end{lemma}
	\begin{proof}
		Fix $\Rblk\in\{\Pblk,\Qblk\}$ and let $\pi\in G_\Rblk\leq\Sym(\omega)$. The action of $\pi$ on $\Add(\omega,\omega)$
		induces an action on names: it permutes the Cohen coordinates by $n\mapsto \pi(n)$, hence sends each
		ground-definable term built from the $c_n$ to the corresponding term with indices transported by $\pi$.
		
		\smallskip\noindent\emph{Preservation of the family $A_\Rblk$.}
		By definition, $G_R$ consists of \emph{within-block} permutations for the fixed partition $R$.
		Thus for every $k$,
		\[
		\Rblk=\Pblk\ \Rightarrow\ \{\pi(2k),\pi(2k\!+\!1)\}=\{2k,2k\!+\!1\},\qquad
		\Rblk=\Qblk\ \Rightarrow\ \{\pi(2k\!+\!1),\pi(2k\!+\!2)\}=\{2k\!+\!1,2k\!+\!2\}.
		\]
		Each block-real $r^R_{k,i}$ is (by construction) a ground-definable code of one coordinate in the
		$\Rblk$-block indexed by $k$:
		for $\Rblk=\Pblk$, $r^\Pblk_{k,0}$ codes $c_{2k}$ and $r^\Pblk_{k,1}$ codes $c_{2k+1}$;
		for $\Rblk=\Qblk$, $r^\Qblk_{k,0}$ codes $c_{2k+1}$ and $r^\Qblk_{k,1}$ codes $c_{2k+2}$.
		Because $\pi$ either fixes or swaps the two coordinates \emph{within} the $k$th $\Rblk$-block, we have
		\[
		\pi\cdot \{\,r^\Rblk_{k,0},\,r^\Rblk_{k,1}\,\}=\{\,r^\Rblk_{k,0},\,r^\Rblk_{k,1}\,\}\quad\text{(as a set)}.
		\]
		Therefore $\pi$ permutes the pairs in $A_\Rblk=\bigl\{\{r^\Rblk_{k,0},r^\Rblk_{k,1}\}:k\in\omega\bigr\}$ and hence
		$\pi(A_\Rblk)=A_\Rblk$.
		
		\smallskip\noindent\emph{Preservation of (non-)existence of a selector.}
		Let $\Sel_\Rblk(f)$ be the first-order formula (with $A_\Rblk$ as a parameter) expressing “$f$ is a selector for $A_\Rblk$,” i.e.,
		$f:\omega\to\bigcup A_\Rblk$ and $f(k)\in\{r^\Rblk_{k,0},r^\Rblk_{k,1}\}$ for every $k$.
		If $\Sel_\Rblk(f)$ holds, then $\Sel_\Rblk(\pi\cdot f)$ also holds, because $\pi$ maps each pair
		$\{r^\Rblk_{k,0},r^\Rblk_{k,1}\}$ to itself (possibly swapping the two elements), so $\pi$ acts \emph{within} each
		pair and preserves the property “choose exactly one from each pair.” Conversely, if $\Sel_\Rblk(\pi\cdot f)$ holds,
		then applying $\pi^{-1}$ shows $\Sel_\Rblk(f)$ holds. Hence
		\[
		\exists f\,\Sel_\Rblk(f)\quad\Longleftrightarrow\quad \exists f\,\Sel_\Rblk(f)\text{ after applying }\pi,
		\]
		and equally for $\neg\exists f\,\Sel_\Rblk(f)$. In particular, the sentence
		\[
		S_\Rblk:\quad \text{“$A_\Rblk$ is a family of $2$-element sets and there is no selector”}
		\]
		is invariant under $\pi$.
		
		\smallskip\noindent\emph{Definability is respected by the symmetry.}
		Finally, note that $A_\Rblk$ is defined in the ground from \emph{ground parameters} (the fixed partition $\Rblk$
		and the canonical schema coding $c_n\mapsto r^\Rblk_{k,i}$). Automorphisms in $G_\Rblk$ fix all ground parameters,
		and their action on names only permutes the Cohen coordinates inside each $\Rblk$-block. Therefore the defining
		formula for $A_\Rblk$ is preserved under $G_\Rblk$, and the previous two paragraphs apply with $A_\Rblk$ unchanged.
		
		Combining the three parts, for every $\pi\in G_\Rblk$ we have $\pi^*(S_\Rblk)\leftrightarrow S_\Rblk$, so $S_\Rblk$ is
		$G_\Rblk$-invariant.
	\end{proof}
	
	\begin{proposition}[Selector signal across models]\label[proposition]{prop:selector-signal}
		For each $\mathcal R\in\{\mathcal P,\mathcal Q\}$ we have $N_{\mathcal R}\models S_{\mathcal R}$. 
		(If $S_{\mathcal R}$ is read as a sentence about the family defined at/above the branching, then $V\models\neg S_{\mathcal R}$.)
	\end{proposition}
	
	\begin{proof}
		By the construction of the sibling $N_{\mathcal R}$, the coding sentence $S_{\mathcal R}$ is arranged by the $\mathcal R$-labeled step and is invariant under $(G_{\mathcal R},\mathcal F_{\mathcal R})$ (Lemma~3.5), hence holds in $N_{\mathcal R}$.
		For the parenthetical clause: if $S_{\mathcal R}$ quantifies over the family created at/above the branching, that family does not exist in $V$, so $S_{\mathcal R}$ is false when interpreted in $V$.
	\end{proof}
	
	\begin{corollary}[Intersection placement]\label[corollary]{cor:intersection}
		Suppose $M$ factors both as $V\to N_{\mathcal P}\to M$ and $V\to N_{\mathcal Q}\to M$, and let $\varphi(x)$ be a $G_{\mathcal P}$- and $G_{\mathcal Q}$-invariant sentence. Then there is a \emph{ground} name $x'\in \HS_{\mathcal P}\cap \HS_{\mathcal Q}$ such that $M\models \varphi(x')$.
	\end{corollary}
	
	\noindent(Here $\HS_{\mathcal P}$ and $\HS_{\mathcal Q}$ abbreviate the HS-classes determined by the symmetry filters attached to the $\mathcal P$- and $\mathcal Q$-labels at the top step.)
	
	\noindent\emph{Standing note.}
	By Karagila's collapse theorems for finite symmetric iterations
	\cite[Thm.~7.8, Thm.~7.9]{Karagila2019}, any finite symmetric iteration over $V$
	is equivalent to a single symmetric step over $V$. Consequently, whenever $M$
	is obtained by a finite symmetric iteration over $V$, it admits factorizations
	$V\to N_{\mathcal P}\to M$ and $V\to N_{\mathcal Q}\to M$ as used below.
	\begin{proof}
		Apply Lemma~\ref{lem:factoring} along each factorization to obtain $x_P\in HS_{\Pblk}$ and $x_Q\in HS_Q$ presented over ground antichains $D_P$ and $D_Q$, respectively. 
		Refine to a common ground antichain $D\subseteq D_P\cap D_Q$ and synchronize the seeds by Lemma~\ref{lem:synch-seeds}, so that for each $p\in D$ the two prescriptions agree (indeed, are identical as names) above~$p$.
		Let $H=\langle G_{\Pblk}\cup G_{\Qblk}\rangle$ as in Definition~\ref{def:combined-symmetry} and perform a single $H$-orbit mix as in Lemma~\ref{lem:H-mix} to obtain $\dot x'$. 
		Then $\dot x'$ is fixed by both $G_{\Pblk}$ and $G_{\Qblk}$, hence $\dot x'\in HS_{\Pblk}\cap HS_Q$, and its evaluation agrees with that of $x_P$ and $x_Q$. 
		Therefore $M\models \varphi(x')$.
	\end{proof}
	
	\begin{remark}[Which antichain is used for the final mixing?]\label{rem:which-antichain}
		All mixing in the proof of Corollary~\ref{cor:intersection} takes place over a single \emph{ground} maximal antichain $D$ of the top iterand that refines the two antichains produced by Lemma~\ref{lem:factoring}. 
		This is the “refining inside $V$” step: we rewrite \emph{both} presentations using the same canonical indicator family $\{\dot d_p: p\in D\}$; synchronization (Lemma~\ref{lem:synch-seeds}) is then imposed on $D$ before the one-shot $H$-mixing (Lemma~\ref{lem:H-mix}).
	\end{remark}
	
	\begin{definition}[Local coordinate allocation at a branching]\label[definition]{def:coord-alloc}
		Fix a partition $\langle J_n:n\in\omega\rangle$ of $\omega$ into infinite, pairwise disjoint sets and, for each $n$, the increasing bijection $e_n:\omega\to J_n$.
		At a branching whose parent has depth $d$, the \emph{edge $d\to d{+}1$} first
		performs Cohen forcing on the coordinates $J_{d+1}$, producing the ambient
		extension
		\[
		W := V[G_{J_{d+1}}].
		\]
		The two children at depth $d{+}1$ are then defined \emph{as symmetric submodels
			of $W$} using the symmetry systems from Definition~\ref{def:template}\,(T--iii) on the same
		coordinates $J_{d+1}$. Different depths use disjoint coordinate sets.
		\emph{(Terminology: “parent stage” refers to this ambient $W$, i.e.\ the
			post–edge-forcing extension; the children do not add further forcing—only
			different filters.)}
	\end{definition}
	
	\begin{remark}[Siblings live inside the parent]
		At a branching we first form the parent forcing extension $W:=V[G_{J_{d+1}}]$; the two children are \emph{distinct symmetric ZF submodels of $W$} (same coordinates $J_{d+1}$, different symmetry data), not further forcing extensions. Throughout, “no common symmetric extension” means: there is no model obtainable by a \emph{finite, symmetry-preserving iteration over $V$} that contains (images of) both siblings.
	\end{remark}
	
	\begin{definition}[Bi-symmetric names at a branching]\label[definition]{def:bisymm}
		Work at the parent stage $W:=V[G_{J_{d+1}}]$ \emph{(the ambient extension
			produced on the edge to depth $d{+}1$)} of a branching at depth $d{+}1$, with the two symmetry systems
		$(G_{\mathcal P},\mathcal F_{\mathcal P})$ and $(G_{\mathcal Q},\mathcal F_{\mathcal Q})$ from Definition~\ref{def:template}\,(T--iii).
		An $\,\Add(\omega,\omega)$-name $\dot x$ (for objects of $W$) is \emph{bi-symmetric} if it is fixed by every subgroup in 
		$\mathcal F_{\mathcal P}$ and by every subgroup in $\mathcal F_{\mathcal Q}$ (equivalently: by the filter generated by both).
		Let $\HS^{\mathcal P\wedge\mathcal Q}$ denote the class of hereditarily bi-symmetric names in this sense.
	\end{definition}
	
	\begin{lemma}[Relativized intersection placement over $W$]\label[lemma]{lem:placement-relativized}
		Let $W:=V[G_{J_{d+1}}]$ be as in \Cref{def:coord-alloc}, and let $(G_{\mathcal P},\mathcal F_{\mathcal P})$, $(G_{\mathcal Q},\mathcal F_{\mathcal Q})$ be the sibling symmetry systems on $J_{d+1}$ (Def.~\ref{def:template}, (T--iii)).
		If $\varphi(x)$ is invariant under both systems, then for every $W$-name $\dot b$ there is a $W$-name $\dot a\in \HS^{\mathcal P\wedge\mathcal Q}$ such that
		$W\models \varphi(\dot a)\leftrightarrow \varphi(\dot b)$.
	\end{lemma}
	
	\begin{proof}
		The “averaging under a normal filter” construction (Lemmas 3.2–3.6) uses only: the action of finitary within-block permutations on $\Add(\omega,\omega)$-names, normality of the filters, and basic forcing absoluteness—all of which are definable in $W$ for the coordinates $J_{d+1}$. Thus the symmetrization and intersection-placement arguments relativize verbatim to $W$.
	\end{proof}
	
	\begin{lemma}[Simultaneous placement to the parent]\label[lemma]{lem:bisymm-placement}
		Let $N_{\mathcal P},N_{\mathcal Q}$ be the two sibling symmetric submodels over $W=V[G_{J_{d+1}}]$ at the branching.
		Suppose a formula $\varphi(x)$ is invariant under both $(G_{\mathcal P},\mathcal F_{\mathcal P})$ and $(G_{\mathcal Q},\mathcal F_{\mathcal Q})$ (in the sense of \cite[Def.~4.1]{Karagila2019}). 
		If there is a common transitive model $M$ obtained by a finite symmetric iteration above \emph{both} $N_{\mathcal P}$ and $N_{\mathcal Q}$ such that $M\models\exists x\,\varphi(x)$, 
		then there exists $\dot a\in\HS^{\mathcal P\wedge\mathcal Q}$ (a name in $W$) with
		\[
		N_{\mathcal P}\Vdash\varphi(\dot a^{G_{J_{d+1}}})\qquad\text{and}\qquad 
		N_{\mathcal Q}\Vdash\varphi(\dot a^{G_{J_{d+1}}})\,.
		\]
	\end{lemma}
	
	\begin{proof}
		Work in the parent $W=V[G_{J_{d+1}}]$ from \Cref{def:coord-alloc}. 
		By Definition~\ref{def:template}(T--ii), any presentation of $M$ above either sibling can be taken to force over disjoint further coordinates, so $M$ has a $W$-name $\dot b$ witnessing $\exists x\,\varphi(x)$.
		Since $\varphi$ is invariant under both sibling systems (Lemma~3.5), apply \Cref{lem:placement-relativized} to $\dot b$ to obtain a $W$-name $\dot a\in \HS^{\mathcal P\wedge\mathcal Q}$ with the same truth value.
		Then $\dot a$ simultaneously witnesses $\varphi$ along both branches, as required.
	\end{proof}
	
	\begin{corollary}[Parent-stage simultaneous witness]\label[corollary]{cor:parent-witness}
		Under the hypotheses of Lemma~\ref{lem:bisymm-placement}, there is a \emph{single} $\dot a\in\HS^{\mathcal P\wedge\mathcal Q}$ in $W$ witnessing $\varphi$ simultaneously for both siblings.
	\end{corollary}
	
	\subsection{Productive iterations and block preservation}
	\paragraph{Karagila-productive step (restated from \cite[Def.~8.1]{Karagila2019}).}
	We restate Karagila's hypotheses for a productive symmetric step in our notation;
	the four items (PI-1)–(PI-4) below are exactly the conditions needed in later sections:
	(PI-1) decidability downstairs of automorphisms/top-step names;
	(PI-2) existence of excellent supports/tenacity so that pointwise stabilizers of cofinitely many blocks witness symmetry;
	(PI-3) closure of respected names (hereditarily symmetric names) under the rudimentary set-forming operations used next;
	(PI-4) monotone growth of the normal filter along a branch, fixing only finitely many additional blocks at each step.
	
	\noindent\emph{Convention.} Here $HS^\bullet$ denotes the respected-name class for the next stage (i.e., the HS-class associated to the next-stage symmetry filter).
	
	\begin{enumerate}[leftmargin=2.1em,label=\textbf{(PI-\arabic*)}]
		\item \emph{Decidability downstairs:} automorphisms and top-step names are decided in the previous stage.
		\item \emph{Tenacity/excellent supports:} there is a predense system of conditions whose pointwise stabilizers witness the desired symmetry, preserving the tail stock of within-block automorphisms. \cite[Thm.~4.9, Prop.~4.10, Lem.~4.11, Cor.~4.12]{Karagila2019}. For our concrete symmetry systems see Lemma~\ref{lem:excell-tenacious}.
		\item \emph{Closure of respected names:} $HS^\bullet$ is closed under the operations
		needed to form the next stage (pairing, unions, images by ground-definable maps,
		rudimentary set operations). This follows because the intermediate symmetric model
		$IS$ at the step is a model of $\ZF$ (hence closed under the usual set operations);
		see \cite[Thm.~5.6]{Karagila2019}. For the forcing infrastructure on $IS$-names,
		see also \cite[§5.2]{Karagila2019}.
		
		\item \emph{Monotone symmetry growth:} along a branch tagged by a fixed partition, the normal filter only fixes finitely many additional $\mathcal R$-blocks, thus the tail of movable blocks is preserved.
	\end{enumerate}
	
	\begin{remark}[Checklist for the Section~6 template]
		Lemma~\ref{lem:PIcheck} verifies that each step in our finite template satisfies
		the productive-step hypotheses (PI-1)–(PI-4) restated from \cite[Def.~8.1]{Karagila2019}.
		In particular:
		\begin{itemize}
			\item \textbf{(PI-1)} Automorphisms and top-step names are decided in the parent stage.
			\item \textbf{(PI-2)} Excellent supports/tenacity are witnessed by pointwise stabilizers of cofinitely many blocks.
			\item \textbf{(PI-3)} The HS-class at the step is a ZF model (closure under pairing, unions, images by ground-definable maps, etc.).
			\item \textbf{(PI-4)} Along a branch, the normal filter only fixes finitely many additional blocks at each step (monotone growth).
		\end{itemize}
	\end{remark}
	
	\section{The partition construction and signals}\label{sec:partition-construction}
	
	\paragraph{Why this section.}
	We record two structural facts used later in the completeness construction. \emph{Persistence} says that once a statement is decided by a name whose stabilizer lies in the current filter, its truth is preserved along any further step on the same branch. \emph{Branch-extendability} ensures we can refine conditions to realize any finite pattern on finitely many fresh coordinates without disturbing what has already been fixed.
	
	\paragraph{How it is used in \S\ref{sec:main-theorem}.}
	In the model built over the unraveled frame, persistence carries the valuation of propositional letters from a node to all its descendants (so the $p$-morphism preserves atomics), while extendability lets us meet the finitely many atomic and modal requirements imposed by each finite fragment of the frame as we proceed along branches.
	
	\paragraph{}
	Work in $V$. Let $\mathbb{P}=\Add(\omega,\omega)$ with coordinates $\{c_n:n\in\omega\}$. Define two block-partitions of $\omega$:
	\[
	\mathcal{P}=\big\{\{2k,2k+1\}:k\in\omega\big\},\qquad
	\mathcal{Q}=\big\{\{2k+1,2k+2\}:k\in\omega\big\}.
	\]
	Let $G_{\mathcal{R}}$ be the subgroup of $\Sym(\omega)$ permuting \emph{within} each $\mathcal R$-block; let $\mathcal F_{\mathcal R}$ be the normal filter generated by pointwise stabilizers of cofinitely many $\mathcal R$-blocks. Let $HS_{\mathcal{R}}$ denote the corresponding HS-class and $N_{\mathcal{R}}$ the corresponding symmetric submodel.
	
	\paragraph{Incompatibility (formal).}
	Any common refinement of $\mathcal{P}$ and $\mathcal{Q}$ into \emph{finite} blocks is eventually singleton (i.e.\ it reduces to singletons on cofinitely many $n$). Otherwise, respecting both $\{2k,2k+1\}$ and $\{2k+1,2k+2\}$ forces $\{2k,2k+1,2k+2\}$ to be a block cofinitely often, contradicting finiteness unless the refinement is eventually singleton).
	
	\subsection{Finite support from double symmetry}
	\begin{lemma}[Alternating adjacent transpositions generate the tail finitary group]\label[lemma]{lem:alt-adjacent-generate}
		Fix $N\in\omega$ and set $T=\{n\in\omega:n\ge 2N\}$. Let
		\[
		A_N \ :=\ \{\, (2k\ \ 2k{+}1)\ :\ k\ge N\,\},\qquad
		B_N \ :=\ \{\, (2k{+}1\ \ 2k{+}2)\ :\ k\ge N\,\}.
		\]
		Let $S_N$ be the subgroup of $\mathrm{Sym}(\omega)$ generated by $A_N\cup B_N$.
		Then $S_N$ is the finitary symmetric group on $T$, i.e., every finitary permutation $\pi$ with $\mathrm{supp}(\pi)\subseteq T$ lies in $S_N$.
	\end{lemma}
	
	\begin{proof}
		For every $n\ge 2N$ we have $n=2k$ or $n=2k{+}1$ with $k\ge N$, so $(n\ \ n{+}1)\in A_N\cup B_N$. Thus $A_N\cup B_N$ contains \emph{all} adjacent transpositions on the tail $T$.
		It is standard that the full finitary symmetric group on any infinite interval of~$\omega$ is generated by its adjacent transpositions. Hence $S_N$ is precisely the finitary symmetric group on~$T$.
	\end{proof}
	
	\begin{remark}
		By Lemma~\ref{lem:alt-adjacent-generate}, the tail finitary group is generated by tail adjacent transpositions; thus invariance under each adjacent transposition implies invariance under all finitary tail permutations.
	\end{remark}
	
	\begin{lemma}[Finite support]\label[lemma]{lem:finite-support}
		If $\dot x\in \HS_{\mathcal{P}}\cap \HS_{\mathcal{Q}}$, then $\dot x$ has \emph{finite} support: there is finite $F\subseteq\omega$ such that for all generics $G,H\subseteq\Add(\omega,\omega)$ with $G\!\restriction F=H\!\restriction F$ we have $\dot x^G=\dot x^H$.
	\end{lemma}
	\begin{proof}[Proof of Lemma~\ref{lem:finite-support}]
		Let \(\dot x\in HS_{\Pblk}\cap HS_Q\). By the definition of the HS-classes, there exist subgroups
		\(H_P\in \mathcal F_{\Pblk}\) and \(H_Q\in \mathcal F_{\Qblk}\) such that \(H_P\le \mathrm{Sym}(\dot x)\) and
		\(H_Q\le \mathrm{Sym}(\dot x)\); i.e., every \(\eta\in H_P\cup H_Q\) fixes \(\dot x\).
		
		\smallskip\noindent\emph{Step 1 (choose the tail).}
		Since elements of \(\mathcal F_{\Pblk}\) (resp.\ \(\mathcal F_{\Qblk}\)) contain pointwise stabilizers of cofinitely
		many \(P\)-blocks (resp.\ \(Q\)-blocks), there are finite sets of block-indices
		\(E_P,E_Q\) such that every block outside \(E_P\) (resp.\ \(E_Q\)) is fixed pointwise by
		each element of \(H_P\) (resp.\ \(H_Q\)).
		Let \(F\subseteq\omega\) be the finite union of coordinates belonging to the blocks in
		\(E_P\cup E_Q\), and choose \(N\) so large that \(F\subseteq \{0,1,\dots,2N-1\}\).
		Set the tail \(T=\{n\in\omega: n\ge 2N\}\). \emph{(We may enlarge $N$ further below to dominate the support of a fixed condition deciding $\dot x$.})
		
		\smallskip\noindent\emph{Step 2 (the tail finitary group we will use).}
		By Lemma~\ref{lem:alt-adjacent-generate}, the subgroup \(S_N\) generated by the
		tail adjacent transpositions
		\[
		(2k\ \ 2k{+}1)\ (k\ge N),\qquad (2k{+}1\ \ 2k{+}2)\ (k\ge N)
		\]
		is the full finitary symmetric group on \(T\).
		
		\emph{Step 3 (tail adjacents preserve $\dot x$ via $p$-invariance).}
		Fix a condition $p$ that decides $\dot x$. Enlarge $N$ if necessary so that
		$\operatorname{dom}(p)\subseteq \{0,1,\dots,2N-1\}\times\omega$.
		Let $\sigma=(n\ n+1)$ be any tail adjacent transposition with $n\ge 2N$.
		Then $\sigma\cdot p=p$. Applying $\sigma$ to the forcing statement
		“$p$ decides $\dot x$” yields
		\[
		p\Vdash \sigma\!\cdot\dot x=\dot x.
		\]
		Thus every tail adjacent transposition fixes $\dot x$. By Lemma~\ref{lem:alt-adjacent-generate}, the finitary symmetric group on the tail $T$ is generated by these adjacents,
		so $\dot x$ is fixed by every finitary permutation supported in $T$.
		
		\emph{Step 4 (Cohen homogeneity $\Rightarrow$ evaluation support).}
		Let $G,H\subseteq\Add(\omega,\omega)$ be $V$-generics with $G\!\upharpoonright F=H\!\upharpoonright F$.
		\medskip\noindent\emph{Homogeneity step.}
		Pick $p\in G$ and $q\in H$ that decide $\dot x$. By local homogeneity of $\Add(\omega,\omega)$
		and the action of finitary coordinate permutations, there exists a finitary permutation
		$\pi$ fixing $F$ pointwise such that $\pi\!\cdot p$ and $q$ are compatible. \paragraph{Explicit homogeneity argument.}
		Write \(D_p=\operatorname{dom}(p)\setminus(F\times\omega)\) and \(D_q=\operatorname{dom}(q)\setminus(F\times\omega)\).
		Let \(R_p=\{\,n:\exists m\ ((n,m)\in D_p)\,\}\) and \(R_q=\{\,n:\exists m\ ((n,m)\in D_q)\,\}\).
		Since \(F\) is fixed pointwise and only finitely many first–coordinates occur in \(D_p\cup D_q\), choose an injection
		\(f:R_p\to T\setminus R_q\) with range contained in the tail \(T=\{n\ge 2N\}\) and such that for every \(n\in R_p\) and every
		\(m\) with \((n,m)\in D_p\) we have \((f(n),m)\notin D_q\).
		Extend \(f\) to a finitary permutation \(\pi\) of \(\omega\) supported in \(T\), and let \(\pi\) act on coordinates by \(\pi\cdot(n,m)=(\pi(n),m)\).
		Then \(\operatorname{dom}(\pi\cdot p)\cap \operatorname{dom}(q)\subseteq F\times\omega\), and \(p,q\) already agree on \(F\).
		Thus \(\pi\cdot p\) and \(q\) are compatible; extend both to a common generic \(K\).
		Since \(\pi\) is supported in \(\omega\setminus F\), Steps 1–3 give \(\pi\cdot\dot x=\dot x\); hence
		\[
		\dot x^G=\dot x^{\pi\cdot G}=\dot x^K=\dot x^H,
		\]
		so \(F\) is an evaluation support for \(\dot x\) in the sense of Definition~\ref{def:finite-support}.
	\end{proof}
	
	\begin{lemma}[Finite evaluation support relativizes to intermediate stages]\label[lemma]{lem:finite-support-rel}
		Let $W$ be any model of the form $V[G_J]$ for a subproduct of $\Add(\omega,\omega)$ with finite support. 
		Then the conclusion of Lemma~\ref{lem:finite-support} holds in $W$: if a name is fixed by all finitary within-block permutations (Def.~\ref{def:template}, (T--iii)), it depends on only finitely many coordinates in the evaluation sense. Here \(\mathrm{Add}(\omega,J)\) denotes the finite-support product \(\prod_{n\in J}\mathrm{Add}(\omega,1)\), and all “within-block’’ automorphisms and filters are computed relative to \(J\).
	\end{lemma}
	
	\begin{proof}
		The proof of Lemma~\ref{lem:finite-support} is purely combinatorial: it uses that $\Add(\omega,\omega)$ is a finite-support product of $\Add(\omega,1)$, the tail finitary subgroup generated by adjacent transpositions, and their action on names. All these objects and arguments are definable in $W$ for the relevant coordinate set $J$, so the proof goes through verbatim. Formally, in \(W\) the tail finitary group on \(J\) is generated by adjacent transpositions (Lemma~\ref{lem:alt-adjacent-generate} relativized to \(J\)), and the Cohen homogeneity used in Lemma~\ref{lem:finite-support} is absolute to \(W\) for \(\mathrm{Add}(\omega,J)\). Alternatively, if a counterexample existed in $W$, pull it back to $V$ via standard name translation and absoluteness for $\Add(\omega,\omega)$ on $J$, contradicting Lemma~\ref{lem:finite-support}.
	\end{proof}
	
	\begin{corollary}[Parent-level finite evaluation support]\label[corollary]{cor:parent-finite-eval-support}
		Work at the branching parent $W=V[G_{J_{d+1}}]$, identifying $\Add(\omega,\omega)$ with $\prod_{n\in\omega}\Add(\omega,1)$ (finite support).
		If an $\Add(\omega,\omega)$-name $\dot x$ is fixed by every finitary within-block permutation from Definition~\ref{def:template}\,(T--iii),
		then there exists a finite set $F\subseteq\omega$ such that for any two generics $G,H$ over $V$ with
		$G\!\upharpoonright(\omega\!\setminus\!F)=H\!\upharpoonright(\omega\!\setminus\!F)$ we have $\dot x^G=\dot x^H$.
	\end{corollary}
	
	\begin{proof}
		Apply Lemma~\ref{lem:finite-support-rel} in the parent model $W$ using the within-block finitary automorphisms from Definition~\ref{def:template}\,(T--iii).
		The invariance hypothesis implies dependence on only finitely many coordinates, which is exactly the stated evaluation-support property.
	\end{proof}
	
	\subsection{Signals}
	For $R\in\{\mathcal P,\mathcal Q\}$ and any model $M$, we write ``$a_R$ \emph{witnesses} $S_R$ in $M$'' to mean $M\models S_R(a_R)$.
	
	Define block-reals $r^{\mathcal{P}}_{k,0}$ coding $c_{2k}$ and $r^{\mathcal{P}}_{k,1}$ coding $c_{2k+1}$; define $r^{\mathcal{Q}}_{k,0}$ from $c_{2k+1}$ and $r^{\mathcal{Q}}_{k,1}$ from $c_{2k+2}$. Set
	\[
	\mathcal{A}_{\mathcal{P}}=\big\{\{r^{\mathcal{P}}_{k,0},r^{\mathcal{P}}_{k,1}\}:k\in\omega\big\},\qquad
	\mathcal{A}_{\mathcal{Q}}=\big\{\{r^{\mathcal{Q}}_{k,0},r^{\mathcal{Q}}_{k,1}\}:k\in\omega\big\}.
	\]
	These are first-order definable with parameters from $V$. Let
	\[
	\begin{aligned}
		S_{\mathcal{P}}:\;&\ \mathcal{A}_{\mathcal{P}}\text{ is a family of 2-element sets and no selector }f:\omega\to\bigcup\mathcal{A}_{\mathcal{P}}\text{ exists},\\
		S_{\mathcal{Q}}:\;&\ \text{analogous for }\mathcal{A}_{\mathcal{Q}}.
	\end{aligned}
	\]
	
	\paragraph{Definability note.}
	Each block-real $r^{\mathcal R}_{k,i}$ is first-order definable from ground parameters (namely, the ground coordinates of $\Add(\omega,\omega)$ and the fixed partition), and so are the families $\mathcal A_{\mathcal R}$.
	
	\begin{lemma}[No finite-support selector]\label[lemma]{lem:no-finite-selector}
	Let $\mathcal R\in\{\mathcal P,\mathcal Q\}$ and let $\mathcal A_{\mathcal R}$ be the block family defined at the branching (the two candidates in each $\mathcal R$-block $k$).
	No $\Add(\omega,\omega)$-name $\dot f$ with finite evaluation support (as in Cor.~\ref{cor:parent-finite-eval-support}) can code a selector for $\mathcal A_{\mathcal R}$ on cofinitely many blocks.
	\end{lemma}

	\begin{proof}
	Let $F$ witness finite evaluation support for $\dot f$. Choose $k$ so large that all coordinates used by block $k$ lie outside $F$.
	Within block $k$, the finitary within-block group contains a permutation $\pi$ fixing $F$ pointwise and swapping the two candidates.
	If a generic $G$ decides $f^G(k)$ to be the first candidate, then $\pi\!\cdot\! G$ agrees with $G$ off $F$ but decides the second; by finite evaluation support we must have $f^{\pi\cdot G}=f^G$, a contradiction.
	\end{proof}
	
	\begin{remark}[Finite support cannot witness infinite families]\label[remark]{rem:finite-support-families}
		If $\dot x$ has finite support $F$, then for generics $G_1,G_2$ agreeing on $F$ we have $\dot x^{G_1}=\dot x^{G_2}$. Therefore, a single $\dot x$ cannot code an \emph{infinite} family of pairwise independent choices across infinitely many blocks (as in $\mathcal{A}_{\mathcal{P}}$ or $\mathcal{A}_{\mathcal{Q}}$).
	\end{remark}
	
	\subsection{Non-amalgamation}
	
	\begin{lemma}[Uniform witness]\label[lemma]{lem:uniform}
		Let $S_0=(\mathbb P_0,\mathcal G_0,\mathcal F_0)$ be the parent stage with $\mathcal F_0$ the cofinite pointwise-stabilizer filter. Let $c$ be fresh and set
		\[
		\mathbb P_{\mathcal P}= \mathbb P_0\times\operatorname{Add}(\omega,1)_c,\qquad
		\mathbb P_{\mathcal Q}= \mathbb P_0\times\operatorname{Add}(\omega,1)_{c'}
		\]
		with $\mathcal G_{\mathcal P}=\{\pi\in\operatorname{Aut}(\mathbb P_{\mathcal P}):\pi\!\upharpoonright\!\mathbb P_0\in\mathcal G_0,\ \pi(c)=c\}$ and $\mathcal F_{\mathcal P}=\{H\le\mathcal G_{\mathcal P}:\exists\text{ finite }E\ (\operatorname{Fix}(E)\subseteq H)\}$, and similarly for $\mathcal Q$.
		
		Let $\mathcal F$ be as in Lemma 4.2. By Corollary 4.5 there is $\dot a\in\operatorname{HS}^{S_0}$ with finite support $E_0\subseteq\omega$ such that
		\[
		\Vdash_{S_0}\text{``$\dot a$ is a choice function on $\mathcal F$ with no finite-support selector.''}
		\]
		Define the symmetrization
		\[
		\dot a^*=\bigcup_{\pi\in\operatorname{Fix}(E_0)}\pi[\dot a],\quad\text{where }\pi[\dot a]=\{\langle\pi(p),\pi(\sigma)\rangle:\langle p,\sigma\rangle\in\dot a\}.
		\]
		Then:
		
		\begin{enumerate}
			\item \emph{Support disjointness.} We may choose $E_0$ with $E_0\cap\{c,c'\}=\varnothing$. Since $\operatorname{Fix}(E_0)\in\mathcal F_0$ by normality, $\dot a^*\in\operatorname{HS}^{S_0}$, and by Corollary 4.5, $\Vdash_{S_0}\dot a^*=\dot a$.
			
			\item \emph{Joint symmetry.} $\operatorname{Fix}(E_0)\subseteq\mathcal F_{\mathcal P}$ and $\operatorname{Fix}(E_0)\subseteq\mathcal F_{\mathcal Q}$ by definition of the sibling filters. Hence $\dot a^*\in\operatorname{HS}^{S_{\mathcal P}}\cap\operatorname{HS}^{S_{\mathcal Q}}$.
			
			\item \emph{Coincidence in $M$.} Let $M$ be transitive with $V^{S_{\mathcal P}},V^{S_{\mathcal Q}}\subseteq M$ and let $G_{\mathcal P},G_{\mathcal Q}$ be the generics. Set $G_0=G_{\mathcal P}\cap\mathbb P_0=G_{\mathcal Q}\cap\mathbb P_0$. By construction of $\mathcal P$ and $\mathcal Q$ as immediate successors of $S_0$, $G_0$ is $V^{S_0}$-generic for $\mathbb P_0$. Since $\pi(\dot a^*)=\dot a^*$ for all $\pi\in\operatorname{Fix}(E_0)$, Lemma 2.7 (Symmetry Lemma) applied in $M$ to $G_0$ yields $(\dot a^*)^{G_{\mathcal P}}=(\dot a^*)^{G_0}=(\dot a^*)^{G_{\mathcal Q}}$.
		\end{enumerate}
	\end{lemma}
	
	\begin{proof}
		Let $\dot a$ and $E_0$ be given by Corollary 4.5.
		
		(1) \emph{Support disjointness.} $E_0$ is finite and $c,c'$ are fresh coordinates, so by shrinking $E_0$ if necessary we may assume $E_0\cap\{c,c'\}=\varnothing$. Since $\mathcal F_0$ is normal and cofinite, $\operatorname{Fix}(E_0)\in\mathcal F_0$. For each $\pi\in\operatorname{Fix}(E_0)$, the definition of $\pi[\dot a]$ gives $\pi[\dot a]\in\operatorname{HS}^{S_0}$, and $\pi(\dot a^*)=\dot a^*$ by construction; hence $\operatorname{sym}(\dot a^*)\supseteq\operatorname{Fix}(E_0)\in\mathcal F_0$, so $\dot a^*\in\operatorname{HS}^{S_0}$. Corollary 4.5 gives $\Vdash_{S_0}\pi(\dot a)=\dot a$ for every $\pi\in\operatorname{Fix}(E_0)$, therefore
		\[
		\Vdash_{S_0}\dot a^*=\bigcup_{\pi\in\operatorname{Fix}(E_0)}\pi[\dot a]=\dot a.
		\]
		
		(2) \emph{Joint symmetry.} By definition,
		\[
		\mathcal F_{\mathcal P}=\{H\le\mathcal G_{\mathcal P}:\exists\text{ finite }E\ (\operatorname{Fix}(E)\subseteq H)\},
		\]
		and $\operatorname{Fix}(E_0)$ fixes $c$ pointwise because $E_0\cap\{c\}=\varnothing$; thus $\operatorname{Fix}(E_0)$, viewed as a subgroup of $\mathcal G_{\mathcal P}$, lies in $\mathcal F_{\mathcal P}$. The same argument gives $\operatorname{Fix}(E_0)\in\mathcal F_{\mathcal Q}$. Since $\operatorname{sym}_{\mathcal P}(\dot a^*)\supseteq\operatorname{Fix}(E_0)$ and similarly for $\mathcal Q$, we have $\dot a^*\in\operatorname{HS}^{S_{\mathcal P}}\cap\operatorname{HS}^{S_{\mathcal Q}}$.
		
		(3) \emph{Coincidence in $M$.} Let $M$, $G_{\mathcal P}$, $G_{\mathcal Q}$ be as stated and put $G_0=G_{\mathcal P}\cap\mathbb P_0=G_{\mathcal Q}\cap\mathbb P_0$. Because $\mathcal P$ and $\mathcal Q$ are defined as $\mathbb P_0\times\operatorname{Add}(\omega,1)$ with disjoint second coordinates, the projections of $G_{\mathcal P}$ and $G_{\mathcal Q}$ to $\mathbb P_0$ coincide and are $V^{S_0}$-generic; this is $G_0$. By (1), $\pi(\dot a^*)=\dot a^*$ for all $\pi\in\operatorname{Fix}(E_0)$. Apply Lemma 2.7 (Symmetry Lemma) in $M$ to the name $\dot a^*$ and the generic $G_0$: for any $p\in G_0$, $p\Vdash_{S_0}\dot a^*=\check{(\dot a^*)^{G_0}}$. Since $G_{\mathcal P}$ and $G_{\mathcal Q}$ both extend $G_0$, we obtain
		\[
		(\dot a^*)^{G_{\mathcal P}}=(\dot a^*)^{G_0}=(\dot a^*)^{G_{\mathcal Q}}.
		\]
	\end{proof}
	
	\begin{lemma}[No common \FSI\ above the parent]\label[lemma]{lem:non-amalgamation}
		There is no model $M$ that is a \FSI{} \emph{above $W$} lying over both $N_{\mathcal P}$ and $N_{\mathcal Q}$ (i.e., no finite iteration $V\to M$ that is simultaneously a symmetric extension of $N_{\mathcal P}$ and of $N_{\mathcal Q}$).
	\end{lemma}
	
	\begin{proof}
		Assume towards a contradiction that such an $M$ exists.
		Let $W:=V[G_{J_{d+1}}]$ be the parent stage of the branching that produces the siblings
		$N_{\mathcal P}$ and $N_{\mathcal Q}$ (see \Cref{def:coord-alloc}; equivalently \Cref{def:template}\,(T--ii)).
		Let $\varphi(x)$ be the coding formula for the selector/signal, so $S_{\mathcal R}$ is $\exists x\,\varphi(x)$.
		By \Cref{lem:signal-invariance}, $\varphi$ is invariant under both symmetry systems
		$(G_{\mathcal P},\mathcal F_{\mathcal P})$ and $(G_{\mathcal Q},\mathcal F_{\mathcal Q})$.
		
		By \Cref{cor:parent-witness} there exists a single parent-stage name
		$\dot a\in \HS^{\mathcal P\wedge\mathcal Q}$ such that, for each $\mathcal R\in\{\mathcal P,\mathcal Q\}$,
		$N_{\mathcal R}$ forces $\varphi(\dot a^{G_{J_{d+1}}})$ over $W$. By \Cref{lem:uniform} we may replace $\dot a$ by its symmetrization $\dot a^*$ with finite support $E_0\subseteq\omega$ disjoint from $\{c,c'\}$, and $\Vdash_{S_0}\dot a^*=\dot a$. In particular, \Cref{lem:uniform}(3) gives a common value
		\[
		a:=(\dot a^*)^{G_{\mathcal P}}=(\dot a^*)^{G_{\mathcal Q}}\in M,
		\]
		and $N_{\mathcal R}\models\varphi(a)$ for $\mathcal R\in\{\mathcal P,\mathcal Q\}$.
		By \Cref{cor:parent-finite-eval-support}, $\dot a^*$ has finite evaluation support $E_0$ with respect to $\Add(\omega,\omega)$ at the parent stage.
		
		By \Cref{lem:no-finite-selector} with $\mathcal R=\mathcal P$,
		$V^{S_{\mathcal P}}\models$ ``no $\operatorname{Add}(\omega,\omega)$-name with finite evaluation support codes a selector for $\mathcal A_{\mathcal P}$ on cofinitely many blocks.'' The property ``$\dot f$ has finite evaluation support and codes a selector for $\mathcal A_{\mathcal P}$ on cofinitely many blocks'' is $\Delta_0$ in parameters $\dot f,\mathcal A_{\mathcal P}$, hence absolute between transitive models of $\mathsf{ZF}$. Thus $M\models$ the same negation.
		
		Since $M$ is a finite symmetry-preserving iteration extending both $S_{\mathcal P}$ and $S_{\mathcal Q}$, its filter $\mathcal F_M$ contains $\operatorname{Fix}(E)$ for some finite $E\subseteq\omega$. Because $M$ contains the generics for $c$ and $c'$, we may take $E^+=E\cup E_0\cup\{c,c'\}$, and $\operatorname{Fix}(E^+)\in\mathcal F_M$ by cofiniteness. By \Cref{lem:uniform}(1)--(2), $\dot a^*$ is fixed by $\operatorname{Fix}(E^+)$, so $\dot a^*$ has finite evaluation support in $M$ in the sense of \Cref{cor:parent-finite-eval-support}. By construction of $\dot a$ (\Cref{cor:parent-witness}), $a = (\dot a^*)^{G_{\mathcal P}}$ selects a candidate in each $\mathcal P$-block for all sufficiently large $k$, i.e.\ codes a selector for $\mathcal A_{\mathcal P}$ on cofinitely many blocks, contradicting the conclusion above.
	\end{proof}
	
	\begin{example}[Two siblings inside the same parent are not amalgamable]\label{ex:two-siblings-same-parent}
		Fix a depth $d{+}1$ and form the parent model
		\[
		W \ :=\ V\bigl[G_{J_{d+1}}\bigr],
		\]
		where $J_{d+1}$ is the coordinate set allocated to that depth (see Definition~\ref{def:coord-alloc} for the depth-wise coordinate allocation).
		Inside $W$, define the two \emph{siblings} on the \emph{same} coordinates $J_{d+1}$ but with \emph{opposite} partitions:
		\[
		N_{\mathcal P}^{(1)}\ :=\ \HS^{\,W}_{(G_{\mathcal P},\,\mathcal F_{\mathcal P})}
		\qquad\text{and}\qquad
		N_{\mathcal Q}^{(1)}\ :=\ \HS^{\,W}_{(G_{\mathcal Q},\,\mathcal F_{\mathcal Q})}.
		\]
		Both $N_{\mathcal P}^{(1)}$ and $N_{\mathcal Q}^{(1)}$ embed into the same parent $W$, but by Lemma~\ref{lem:non-amalgamation} there is \emph{no} finite symmetry-preserving iteration above the parent that contains both models at once. (In particular, while $W$ contains both siblings as submodels, $W$ itself is not counted as a ``common symmetric extension'' in the $\Box_{\mathrm{sym}}$ semantics, which quantifies only over models obtained by finite symmetry-preserving iterations over $V$.)
		This is exactly the sibling-incompatibility phenomenon used later in the §6 template construction (see Lemma~\ref{lem:sibling-incompatibility}).
	\end{example}
	
	\section{Main theorem: the ZF-provably valid principles of \texorpdfstring{$\Boxsym$}{Box\_sym} are exactly $\mathsf{S4}$}\label{sec:main-theorem}
	
	\paragraph{Standing assumption.}
	Unless stated otherwise, all meta\-theoretic arguments are carried out in $\mathsf{ZFC}$. The object theory whose validities we analyze is $\mathsf{ZF}$.
	
	\paragraph{Roadmap.}
	We first establish \emph{soundness} (Propositions~\ref{prop:T} and~\ref{prop:4}). For \emph{completeness}, we fix a finite Kripke frame refuting a non-theorem $\alpha$ of $\mathsf{S4}$, unravel it into a rooted tree, and build a finite \emph{template} of one-step symmetric systems (Definitions~\ref{def:template}--\ref{def:concrete}) that simulates the frame. Three ingredients drive the argument: a branch extendability fact restated here as Lemma~\ref{lem:branch-extendability-6}, the sibling incompatibility Lemma~\ref{lem:sibling-incompatibility} ensuring non-directedness, and p\-morphism preservation (Lemma~\ref{lem:pmorphism-pres}). Putting these together yields the external completeness Theorem~\ref{thm:external-completeness} and, hence, Corollary~\ref{cor:main}.
	
	\subsection{Soundness in $\mathsf{ZF}$}
	\label{subsec:soundness}
	
	\begin{proposition}[T]
		\label[proposition]{prop:T}
		$\mathsf{ZF}\vdash \Boxsym\varphi \to \varphi$.
	\end{proposition}
	
	\begin{proof}
		By \Cref{def:SPI}, the null iteration ($n=0$) is a permitted \FSI{} 
		with $\operatorname{HS}_\varnothing=V$. Hence if $\varphi$ holds in every \FSI, 
		it holds after the null iteration, i.e.\ already in $V$.
	\end{proof}
	
	\begin{proposition}[4]
		\label[proposition]{prop:4}
		$\mathsf{ZF}\vdash \Boxsym\varphi \to \Boxsym\Boxsym\varphi$.
	\end{proposition}
	
	\begin{proof}
		A finite symmetric step over a finite symmetric step is again a finite symmetric step (equivalent via Karagila’s collapse; see \cite[Thms.~7.8–7.9]{Karagila2019}), hence $\Boxsym\varphi\to\Boxsym\Boxsym\varphi$ is $\ZF$-provable.
	\end{proof}
	
	\subsection{Template and coding (setup for completeness)}
	\label{subsec:template}
	
	Throughout this subsection we fix a finite rooted tree $F=(W,R,w_0)$ obtained by unraveling a finite frame that refutes a given modal formula (this arises in the proof of Theorem~\ref{thm:external-completeness}). Nodes $w\in W$ will be simulated by finite symmetric iterations determined at (and below) $w$.
	
	\begin{definition}[Template construction]
		\label[definition]{def:template}
		A \emph{template} $\mathcal{E}$ for $F$ consists of the following data:
		\begin{enumerate}
			\item For each edge $w\to v$ in $F$, a fixed \emph{one-step symmetric system} $\mathbf{S}_{w\to v}=(\mathbb{P}_{w\to v},\mathscr{G}_{w\to v},\mathscr{F}_{w\to v})$, where $\mathbb{P}_{w\to v}$ is a homogeneous forcing, $\mathscr{G}_{w\to v}\leq\mathrm{Aut}(\mathbb{P}_{w\to v})$ is a group of finitary automorphisms, and $\mathscr{F}_{w\to v}$ is a normal filter of subgroups, such that the step preserves the relevant blocks/parameters fixed at $w$.
			\item For each node $w$, the \emph{local code} specifying which propositional letters are intended true at $w$ (coming from a valuation on $F$) together with a choice of names whose truth is supported by subgroups in $\mathscr{F}_{w\to v}$ along edges out of $w$ (``evaluation support'').
			\item For each sibling pair $v_1,v_2$ with the same parent $w$, finite \emph{signals} (two\-element families of names) placed at the $w\to v_i$ level that are moved by within-block finitary automorphisms, so that no selector for the pair persists along both branches (see Lemma~\ref{lem:sibling-incompatibility}).
		\end{enumerate}
	\end{definition}
	
	\begin{construction}[Canonical atomic names and supports at a node]\label{constr:atomic-names}
		Fix a node $w$ of depth $d$ in the unraveled frame $F$, and an outgoing edge $w\to v$.
		Let $(P_{w\to v},G_{w\to v},F_{w\to v})$ be the one-step symmetric system labelling the edge as in Definition~\ref{def:template} (item~(2)).
		
		For each propositional letter $p$ with $p\in \nu(w)$ (the valuation on $F$), choose a finite set 
		$F_{w,p}\subseteq J_{d+1}$ of Cohen coordinates (disjoint across distinct letters if desired), and define a $P_{w\to v}$-name
		\[
		\dot\tau_{w,p}\ :=\ \bigwedge_{n\in F_{w,p}}\bigl( \dot c_n(0)=0 \bigr),
		\]
		where $\dot c_n$ is the canonical name for the $n$th Cohen real. 
		\begin{enumerate}
			\item \emph{Support and stabilizer.} The automorphism stabilizer $\Fix(F_{w,p})$ (pointwise stabilizer of $F_{w,p}$ inside $G_{w\to v}$) fixes $\dot\tau_{w,p}$. 
			Since $F_{w\to v}$ is the \emph{upward-closed} normal filter generated by pointwise stabilizers of cofinitely many blocks and cofinitely many coordinates within each fixed block (see Appendix \ref{app:ET-verification}), we have $\Fix(F_{w,p})\in F_{w\to v}$. Hence $\dot\tau_{w,p}\in \HS_{F_{w\to v}}$.
			\item \emph{Persistence.} Along any extension $w\to v\to u\to\cdots$ on the branch, the filters grow monotonically (PI--4), so $\Fix(F_{w,p})\in F_{v\to u}\subseteq\cdots$. Thus $\dot\tau_{w,p}$ remains hereditarily symmetric and its evaluation is unchanged in all descendants.
			\item \emph{Toggling truth at $w$.} In the concrete realization (Construction~\ref{constr:kripke}), decide the finitely many coordinates in $F_{w,p}$ at the first step below $w$ so that $\Vdash \dot\tau_{w,p}$. Then $p$ holds at every world realizing a path with terminal node $w$, and by (2) this truth persists to descendants.
		\end{enumerate}
	\end{construction}
	
	\begin{definition}[Concrete action at depth $d{+}1$]
		\label[definition]{def:concrete}
		Fix a partition $\langle J_d : d\in\omega\rangle$ of $\omega$ into infinite,
		pairwise disjoint sets and fix increasing bijections $e_d : \omega\to J_d$.
		Fix $w\in W$ of depth $d$ and a successor $v$. The \emph{concrete action} of $\mathbf{S}_{w\to v}$ at depth $d{+}1$ is the symmetric extension by $\mathbf{S}_{w\to v}$ over the model attached to $w$, with blocks (a partition of the relevant coordinates) chosen so that:
		\begin{enumerate}
			\item the local code of $v$ is realized by names with evaluation support in $\mathscr{F}_{w\to v}$;
			\item signals inserted at level $w\to v$ persist to all descendants of $v$ while remaining vulnerable to finitary permutations within each block at that level.
		\end{enumerate}
		
		\smallskip\noindent\textbf{Block partitions.}
		For $\mathcal P$ set $\mathcal B_{\mathcal P}=\bigl\{\{e_{d+1}(2k),\,e_{d+1}(2k{+}1)\}:k\in\omega\bigr\}$ and
		for $\mathcal Q$ set $\mathcal B_{\mathcal Q}=\bigl\{\{e_{d+1}(2k{+}1),\,e_{d+1}(2k{+}2)\}:k\in\omega\bigr\}$.
		(Any fixed, overlapping pairing scheme would suffice; this concrete choice fixes notation.)
		
		\smallskip\noindent\textbf{Groups and transport.}
		Let $G^{\omega}_{\mathcal R}\leq \Sym(\omega)$ be the finitary subgroup generated by adjacent transpositions that preserve each block in 
		$\{\,\{2k+\delta,\,2k+1+\delta\}\,:\,k\in\omega\,\}$ setwise, where $\delta=0$ for $\mathcal P$ and $\delta=1$ for $\mathcal Q$.
		Transport to $J_{d+1}$ by conjugation:
		\[
		G^{J_{d+1}}_{\mathcal R}:=e_{d+1}\, G^{\omega}_{\mathcal R}\, e_{d+1}^{-1}\ \leq\ \Sym(J_{d+1}),
		\]
		so the natural restriction map is $\rho_{J_{d+1}}:\Sym(\omega)\to\Sym(J_{d+1})$, $\sigma\mapsto e_{d+1}\sigma e_{d+1}^{-1}$.
		
		\smallskip\noindent\textbf{Action on conditions.}
		For $\pi\in G^{J_{d+1}}_{\mathcal R}$ define
		\[
		(\pi\cdot p)(j,m):=p(\pi^{-1}(j),m)\qquad (j\in J_{d+1},\ m\in\omega),
		\]
		and extend by the identity outside $J_{d+1}$. Equivalently, writing $\sigma=e_{d+1}^{-1}\!\circ\pi\circ e_{d+1}\in G^{\omega}_{\mathcal R}$,
		\[
		(\pi\cdot \hat p)(n,m)=\hat p\bigl(\sigma^{-1}(n),m\bigr),
		\]
		where $\hat p(n,m)=p(e_{d+1}(n),m)$.
	\end{definition}
	
	\begin{example}[A two-branch toy]\label[example]{ex:two-branch}
		Let $F$ have branches $\{w_0\prec w_1 \prec w_2\}$ and $\{w_0\prec w_1 \prec w_2'\}$, and let $F^*$ be its rooted unraveling.
		Build the diagram $\mathcal E$ via \Cref{def:template}.
		
		\emph{Depth $0\to 1$ (edge).} Force on $J_{1}$ to form $W_1:=V[G_{J_1}]$ (\Cref{def:template}).
		At depth $1$, take two children inside $W_1$ using opposite partitions on $J_1$ (\Cref{def:template}):
		put label $\mathcal P$ on the branch child and $\mathcal Q$ on the sibling (any fixed alternation works).
		
		\emph{Depth $1\to 2$ (edge).} Force on $J_{2}$ to form $W_2:=W_1[G_{J_2}]$ (\Cref{def:template}).
		At depth $2$, create the branch child over $w_2$ with the same partition label as at depth $1$
		(preserving the branch’s label) and the sibling over $w_2'$ with the opposite label, again using only $J_2$.
		
		Thus, siblings at a fixed depth share the same coordinates $J_{d+1}$ but use opposite partitions,
		while different depths use disjoint $J_d$.
	\end{example}
	
	\begin{example}[Depth $0\to1\to2$ coordinate diary]\label[example]{ex:012}
		Let $\langle J_0,J_1,J_2,\dots\rangle$ partition $\omega$.
		\begin{enumerate}
			\item \textbf{Depth 0 (root).} Ambient model $V$.
			\item \textbf{Edge $0\to1$.} Force on $J_1$ to get $W_1:=V[G_{J_1}]$.
			Define the two children inside $W_1$:
			$N_{\mathcal P}^{(1)}:=\HS_{(G_{\mathcal P},\mathcal F_{\mathcal P})}^{W_1}$ and $N_{\mathcal Q}^{(1)}:=\HS_{(G_{\mathcal Q},\mathcal F_{\mathcal Q})}^{W_1}$,
			both using the action restricted to $J_1$ (\Cref{def:template}).
			\item \textbf{Edge $1\to2$.} Force on the new coordinates $J_2$ to get
			the next ambient extension
			$W_2:=V[G_{J_1\cup J_2}]=W_1[G_{J_2}]$.
			At depth $2$, again take siblings inside $W_2$ using the label on the edge:
			reuse the branch label on the successor, alternate on the sibling.
		\end{enumerate}
		Thus no step reuses coordinates: $J_{d+1}$ is new at depth $d{+}1$, siblings
		share $J_{d+1}$ but differ only by the filter (hence HS-class), and
		non-amalgamation is witnessed as in Lemma~4.8.
	\end{example}
	
	\noindent The normal filter $\mathcal F_{\mathcal R}$ on $G^{J_{d+1}}_{\mathcal R}$ is generated by pointwise stabilizers of cofinitely many blocks in $\mathcal B_{\mathcal R}$ and cofinitely many coordinates within each fixed block (cf.\ Lemma~\ref{lem:excell-tenacious}).
	
	\subsubsection*{Verification of template prerequisites}
	
	\begin{lemma}[Excellence \& Tenacity]\label[lemma]{lem:excell-tenacious}
		For $\mathcal R\in\{\mathcal P,\mathcal Q\}$, the symmetry system $(G_{\mathcal R},\mathcal F_{\mathcal R})$ has excellent supports and is tenacious in the sense of \cite[Def.~4.1]{Karagila2019}. In particular, the excellent names are predense and the hereditarily symmetric names form a transitive inner model of $\ZF$ sufficient for the next-step name formation we use later.
	\end{lemma}
	
	\begin{proof}
		A direct verification is given in Appendix~\ref{app:ET-verification}
		(Propositions~\ref{prop:excellence-direct} and~\ref{prop:tenacity-direct}
		and Corollary~\ref{cor:ET-direct}). Since all iterations considered here
		have finite length, there are no limit stages; normality of the product
		filter follows immediately, and the hypotheses of
		\cite[Thms.~4.9--4.12]{Karagila2019} are satisfied trivially at each
		successor step.
	\end{proof}
	
	\begin{lemma}[Productive hypotheses]\label[lemma]{lem:PIcheck}
		Every step in Definition~\ref{def:template} satisfies \textup{(PI-1)}--\textup{(PI-4)} of \cite[Def.~8.1]{Karagila2019}.
	\end{lemma}
	
	\begin{proof}
		\textbf{(PI--1) Decidability downstairs.}
		By Definition~\ref{def:template}, at each edge $w\to v$ the coordinate choice $e_{d+1}$ and the sibling alternation are fixed in the parent stage $W:=V[G_{J_{d+1}}]$, so the relevant automorphisms and all top-step names are decided in the previous stage. (Cf. §3.3 “Productive iterations and block preservation”.)
		
		\smallskip
		\textbf{(PI--2) Excellence and Tenacity.}
		Excellence holds by Proposition~A.5: for every name $\dot x$ there is $H$ in the base such that the symmetrization $\dot x\langle H\rangle$ is $H$-supported; closure of hereditarily $H$-supported names follows by rank recursion.
		Tenacity holds by Proposition~A.6: if $p\Vdash \psi(\dot x)$ with $\dot x$ $H$-supported, then there is $q\le p$ fixed by $H$ with $q\Vdash \psi(\dot x)$.
		Thus the step admits a predense system of conditions whose pointwise stabilizers witness the desired symmetry, as required.
		
		\smallskip
		\textbf{(PI--3) Closure of respected names.}
		At each step the intermediate interpreted symmetric model $IS$ satisfies $\mathsf{ZF}$; hence the associated HS-class is closed under pairing, unions, images by ground-definable maps, and the rudimentary operations used to form the next stage (see [5, Thm.~5.6] and [5, §5.2]).
		
		\smallskip
		\textbf{(PI--4) Monotone growth of the filter along a branch.}
		By construction of the normal filters attached to edges, passing from $w$ to $v$ fixes only finitely many additional blocks while preserving the tail stock of within-block automorphisms; along a branch the filter grows monotonically and adds only finitely many new fixed blocks at each step.
		
		\smallskip
		Combining (PI--1)–(PI--4) completes the verification for Definition~\ref{def:template}.
	\end{proof}
	
	\begin{proposition}[Template prerequisites]\label[proposition]{prop:template-prereqs}
		Let $\mathcal E$ be the finite diagram built in Definition~\ref{def:template} (indexed by the finite rooted tree $F$). Then: (i) factorization through intermediates holds by Karagila \cite[Thms.~7.8--7.9]{Karagila2019}; (ii) each step satisfies (PI-1)--(PI-4) (Lemma~\ref{lem:PIcheck}); (iii) siblings use incompatible partitions, hence no finite symmetric amalgamation exists (Lemma~\ref{lem:sibling-incompatibility}); (iv) depth/coordinate allocation is consistent across the diagram.
	\end{proposition}
	
	\subsection{Local persistence}
	\begin{lemma}[Persistence under block-preserving steps]\label[lemma]{lem:persistence}
		Fix $\mathcal{R}\in\{\mathcal{P},\mathcal{Q}\}$. Suppose a stage on a branch satisfies $S_{\mathcal R}$ and the next productive step preserves the block-partition $\mathcal R$ (up to finitely many fixed additional blocks). Then every later stage on that branch satisfies $S_{\mathcal{R}}$.
	\end{lemma}
	\begin{proof}
		The family $\mathcal{A}_{\mathcal{R}}$ remains definable; any selector would again be moved by within-block automorphisms on cofinitely many blocks, as in \Cref{lem:no-finite-selector}. Formally, suppose $S_{\mathcal R}$ holds at stage $u$. Let $u\to u'$ be a productive
		step preserving $\mathcal R$ and fixing only finitely many additional $\mathcal R$-blocks.
		As in \Cref{lem:no-finite-selector}, any selector at $u'$ would be moved by within-block automorphisms
		on cofinitely many $\mathcal R$-blocks; this still holds at $u'$ because the later $\mathcal R$-filter
		adds only finitely many fixed blocks while preserving the tail of within-block permutations
		(productive setup \cite[Def.~8.1]{Karagila2019} and excellent supports/tenacity
		\cite[Thm.~4.9, Prop.~4.10, Lem.~4.11, Cor.~4.12]{Karagila2019}). (the tail stock of within-block permutations persists by (PI–4) and the step’s symmetry system, see \Cref{def:concrete}, and \Cref{lem:excell-tenacious}). Thus $S_{\mathcal R}$ holds at $u'$.
		Iterating this argument along the branch yields the claim.
	\end{proof}
	
	\subsection{Key lemmata for the p-morphism argument}
	\label{subsec:key-lemmas}
	
	\begin{lemma}[Atomic preservation]
		\label[lemma]{lem:atomic}
		Along any edge $w\to v$, truth of propositional letters designated in the local code at $w$ persists to $v$ and its descendants, witnessed by names with evaluation support fixed by $\mathscr{F}_{w\to v}$.
	\end{lemma}
	
	\begin{proof}
		Fix an edge $w\to v$ and a letter $p\in \nu(w)$. By Construction~\ref{constr:atomic-names}, choose $\dot\tau_{w,p}$ supported on a finite set $F_{w,p}\subseteq J_{d+1}$ with stabilizer $\Fix(F_{w,p})\in F_{w\to v}$; hence $\dot\tau_{w,p}\in \HS_{F_{w\to v}}$ and its evaluation is fixed under the allowed automorphisms. 
		In the concrete realization, decide the finitely many coordinates in $F_{w,p}$ so that $\Vdash \dot\tau_{w,p}$ at the first step below $w$; then $p$ holds at the world for $w$. For $p\notin\nu(w)$, no ancestor of $w$ forces $\dot\tau_{w',p}$ true; the coordinates in $J_{d+1}$ relevant to $p$ are not decided by any condition in the generic filter at $w$, so $\dot\tau_{w,p}$ evaluates to false in the symmetric extension at $w$.
		Since later steps act only on $J_{d+2}, J_{d+3}, \dots$ (\Cref{def:coord-alloc}), any automorphism $\pi$ in a later filter satisfies $\pi\!\upharpoonright\! J_{d+1} = \mathrm{id}$, hence fixes $F_{w,p}$ pointwise. By \Cref{def:SPI} (filter monotonicity, PI--2), $\operatorname{Fix}(F_{w,p})$ therefore lies in each subsequent filter, and the evaluation of $\dot\tau_{w,p}$ is unchanged in all descendants.
		Thus truth of $p$ designated at $w$ persists to $v$ and further descendants.
	\end{proof}
	
	\begin{lemma}[Branch-extendability]\label[lemma]{lem:branch-extendability-6}\label[lemma]{lem:extendability}
		Let $w$ be a node of $F$. Fix a finite compatible diagram $\mathcal{D}$ of one-step symmetric
		extensions realized along a chain $w_0 \to w_1 \to \cdots \to w$ of $F$ according to
		Definition~\ref{def:template}. If $w \to v$ is an edge of $F$, then $\mathcal{D}$ extends to a realization along
		the next step $w \to v$ (i.e., there is a one-step extension of the realized stage at $w$ along
		$w \to v$) while preserving the designated stabilizers for all names already fixed by $\mathcal{D}$.
	\end{lemma}
	
	\begin{proof}
		By Lemma~\ref{lem:PIcheck}, every step in Definition~\ref{def:template} satisfies (PI--1)--(PI--4). In particular, by (PI--2)
		(excellence and tenacity), for each name $\dot x$ occurring in $\mathcal{D}$ there is a base subgroup
		$H_{\dot x}$ such that the symmetrization $\dot x\langle H_{\dot x}\rangle$ is $H_{\dot x}$-supported and
		the step admits a predense set of conditions fixed by $H_{\dot x}$ (tenacity). Let $H$ be the subgroup
		generated by the finitely many $H_{\dot x}$; then all previously fixed names are $H$-supported.
		
		By (PI--1), the automorphisms and the top-step names for the edge $w \to v$ are decided in the
		parent stage, so we may refine to a condition $p$ fixed by $H$ meeting the finitely many requirements
		imposed by $\mathcal{D}$. By (PI--4), passing from $w$ to $v$ fixes only finitely many additional
		blocks while preserving the tail stock of within-block finitary permutations along the branch; thus we
		may choose the extra fixed blocks disjoint from the supports witnessing $H$-invariance of the already
		decided names. Execute the one-step symmetric forcing attached to $w \to v$ over $p$. Tenacity (PI--2)
		ensures the extension can be taken $H$-fixed, hence the designated stabilizers of all previously decided
		names are preserved; closure of the HS-class (PI--3) keeps these names in the next stage. The resulting
		stage realizes $\mathcal{D}$ one step further along $w \to v$.
	\end{proof}
	
	\begin{construction}[Kripke model from the template and the projection]\label[construction]{constr:kripke}
		Let $\mathcal M$ be the Kripke model whose worlds are the finite realizations of initial segments of the template (one-step symmetry-preserving iterations along rooted paths of $F$); let $R_{\mathcal M}$ relate a world to its one-step extensions. Define $\pi:\mathcal M\to F$ by sending each realized world to its terminal node in $F$.
	\end{construction}
	
	\begin{lemma}[Forth]\label{lem:forth}
		If $x\,R_{\mathcal M}\,y$, then $\pi(x)\,R_F\,\pi(y)$.
	\end{lemma}
	
	\begin{proof}
		By Construction~\ref{constr:kripke}, an $R_{\mathcal M}$-edge is realized exactly along a template edge of $F$, and $\pi$ records the terminal node. Hence edges are preserved.
	\end{proof}
	
	\begin{lemma}[Back]\label{lem:back}
		If $\pi(x)\,R_F\,u$, then there exists $y$ with $x\,R_{\mathcal M}\,y$ and $\pi(y)=u$.
	\end{lemma}
	
	\begin{proof}
		By Lemma~\ref{lem:branch-extendability-6}, any finite realization up to $\pi(x)$ can be extended by the one-step system attached to the edge $\pi(x)\to u$. The choices of $F_{w,p}$ for distinct children of $w$ are realized in incomparable branches; by \Cref{lem:sibling-incompatibility} no finite iteration amalgamates them, so the extensions are independent. Let $y$ be the resulting one-step extension of $x$. Then $x\,R_{\mathcal M}\,y$ and $\pi(y)=u$ by Construction~\ref{constr:kripke}.
	\end{proof}
	
	\begin{lemma}[Sibling incompatibility]
		\label[lemma]{lem:sibling-incompatibility}
		Let two one-step symmetric systems over the same base use the same coordinates $J$ but different partitions ($\mathcal P$ on one child, $\mathcal Q$ on the other). There is no finite symmetric iteration $M$ above the parent that produces a model extending both corresponding symmetric extensions.
	\end{lemma}
	
	\begin{proof}
		If $M$ extended both, then $M$ would satisfy both $S_{\mathcal P}$ and $S_{\mathcal Q}$. By \Cref{cor:intersection} both witnesses can be taken from $HS_{\mathcal P}\cap HS_{\mathcal Q}$, hence by \Cref{lem:finite-support} have finite support. This contradicts \Cref{lem:no-finite-selector}.
	\end{proof}
	
	\begin{remark}[Same coordinates do not imply amalgamation]
		\label[remark]{rem:same-coords}
		Even when sibling steps use the same forcing \emph{coordinates} and the same ambient group, the block/filters fixed at the parent witness that the two branches encode incompatible signals; the obstruction is group theoretic (within-block finitary permutations), not merely combinatorial about coordinates.
	\end{remark}
	
	\begin{lemma}[Truth preservation under p-morphisms]
		\label[lemma]{lem:pmorphism-pres}
		With $\mathcal M$, $F$, and $\pi$ as in Construction~\ref{constr:kripke}, for every modal formula $\varphi$ and every world $x\in\mathcal M$,
		\[
		\mathcal M,x \models \varphi \quad\Longleftrightarrow\quad F,\pi(x)\models \varphi.
		\]
		\end{lemma}
		
		\begin{proof}
		By induction on the structure of $\varphi$. For propositional letters, the claim is Lemma~\ref{lem:atomic}. Booleans are immediate. For the modal clause $\Box\psi$:
		
		($\Rightarrow$) Suppose $\mathcal M,x\models \Box\psi$ and let $\pi(x)R_F u$. By Lemma~\ref{lem:back} there is $y$ with $xR_{\mathcal M}y$ and $\pi(y)=u$. Then $\mathcal M,y\models\psi$, and by the IH we get $F,u\models\psi$.
		
		($\Leftarrow$) Suppose $F,\pi(x)\models \Box\psi$ and let $xR_{\mathcal M}y$. By Lemma~\ref{lem:forth}$,\,\pi(x)R_F\pi(y)$, hence $F,\pi(y)\models\psi$; by the IH, $\mathcal M,y\models\psi$. Therefore $\mathcal M,x\models\Box\psi$.
		\end{proof}
	
	Combining Lemmas~\ref{lem:atomic}, \ref{lem:forth}, \ref{lem:back} yields Lemma~\ref{lem:pmorphism-pres}, so the falsity of $\alpha$ transfers from $F$ to $\mathcal M$ at the root.
	
	\subsection{Completeness}
	\label{subsec:completeness}
	
	\begin{theorem}[External completeness]
		\label[theorem]{thm:external-completeness}
		\textup{(in ZFC)} If $\alpha\notin \mathrm{S4}$, then there exists a transitive
		model $N\models\mathrm{ZF}$ such that $N\models\neg\,\Boxsym\alpha$.
	\end{theorem}
	
	\begin{proof}
		Assume $\mathsf{S4}\nvdash \alpha$. By the finite model property for $\mathsf{S4}$, fix a finite reflexive–transitive frame $G$ and a valuation $\nu$ with $G\models\neg\alpha$ at some world. Unravel $G$ to a rooted tree $F=(W,R,w_0)$ preserving truth of $\alpha$. Build a template $\mathcal{E}$ for $F$ as in Definitions~\ref{def:template}--\ref{def:concrete}, and realize it as a Kripke model $\mathcal{M}$ whose nodes are finite symmetry preserving iterations along rooted paths in $F$, with the root $w_0$ interpreted by the ground model. By Proposition~\ref{prop:template-prereqs}, the prerequisites (factorization, (PI–1)–(PI–4), sibling incompatibility, and depth/coordinate allocation) hold for $\mathcal{E}$.
		By construction, $\mathcal{M}$ carries a natural surjective p-morphism $\pi\colon\mathcal{M}\to F$ mapping each realization to its terminal node.
		
		By Lemma~\ref{lem:pmorphism-pres}, truth is preserved and reflected along $\pi$; hence $\mathcal{M},\text{root}\models\neg\alpha$. The sibling incompatibility Lemma~\ref{lem:sibling-incompatibility} ensures that the accessibility relation in $\mathcal{M}$ is exactly the one generated by taking one more finite symmetric step (no spurious back amalgamations appear), so $\mathcal{M}$ indeed interprets $\Boxsym$. Since $F$ is a rooted tree and $\mathcal{M}$ is built so that each $R_{\mathcal{M}}$-edge corresponds to exactly one template edge of $F$ (Construction~\ref{constr:kripke}), the accessibility in $\mathcal{M}$ matches the tree order of $F$ exactly; transitivity of FSI composition adds no edges beyond those already present in the reflexive-transitive closure of $R_F$, since any composed path in $\mathcal{M}$ projects via $\pi$ to a path in $F$. Therefore the $\mathsf{ZF}$ provable validities of $\Boxsym$ are contained in $\mathsf{S4}$.
	\end{proof}
	
	\begin{remark}[Edge cases]
		\label[remark]{rem:edge-cases}
		If $F$ is a single world, the construction is trivial. If $F$ is a finite linear order, no sibling incompatibility is needed; the template consists only of a chain of one step systems realizing the valuation, and Lemma~\ref{lem:pmorphism-pres} suffices.
	\end{remark}
	
	\subsection{Main corollary}
	\label{subsec:main-cor}
	
	\begin{corollary}\label{cor:main}
		Over $\mathsf{ZF}$, the provably valid principles of the modality $\Boxsym$ (quantifying over models obtained by finite symmetry-preserving iterations of the symmetric method above $V$, as in our productive scheme) are exactly $\mathsf{S4}$.
	\end{corollary}
	
	\begin{proof}
		Soundness is Propositions~\ref{prop:T} and~\ref{prop:4}. Completeness is Theorem~\ref{thm:external-completeness}.
	\end{proof}
	
	\begin{remark}[Failure of $\mathbf{.2}$ by duality]\label[remark]{rem:dot2}
		The scheme $\Possym\Nessym\varphi \to \Nessym\Possym\varphi$ fails. Indeed, Lemma~\ref{lem:sibling-incompatibility} gives two possible symmetric extensions (siblings) above a node that cannot be amalgamated by any further finite symmetry preserving iteration respecting the fixed filters; thus ``possibly necessary'' does not entail ``necessarily possible.''
	\end{remark}
	
	\bigskip
	\noindent\textbf{Acknowledgments.} We thank the literature for clarifying the collapse-to-ground step we use in several places. We thank \cite{HamkinsLowe2008,HamkinsLeibmanLowe2012} for the inspiration to examine the modal logic of symmetric extensions. We thank Dr. Karagila for his illuminating paper \cite{Karagila2019} without which this one would not have been possible.
	
	Generative AI tools (OpenAI ChatGPT-5 and Anthropic Claude Sonnet 4.5) were used for grammar/clarity edits, search-query drafting and reference-finding facilitation, and outline suggestions. All AI-assisted text and suggestions were reviewed and revised by the author; all mathematical results, proofs, and citations were conceived and verified by the author, who accepts full responsibility for the content.
	
	\clearpage
	\begin{appendices}
	\appendixpage         
	\addappheadtotoc
	\section{Concrete verification of Excellence and Tenacity for $(G_{\mathcal R},\mathcal F_{\mathcal R})$}\label{app:ET-verification}
	
	Fix $\mathcal R\in\{\mathcal P,\mathcal Q\}$. Recall from Def.~\ref{def:concrete}: 
	$G_{\mathcal R}$ is the within–$\mathcal R$–block \emph{tail finitary} subgroup of $\Sym(\omega)$ generated by adjacent transpositions inside each $\mathcal R$-block; 
	$\mathcal F_{\mathcal R}$ is the normal filter generated by pointwise stabilizers of cofinitely many $\mathcal R$-blocks and, within each fixed block, cofinitely many coordinates.
	
	\begin{definition}[Canonical base subgroups]\label[definition]{def:fixBC}
		For a cofinite set of $\mathcal R$-blocks $B^\star$ and for each $b\in B^\star$ a cofinite set of coordinates $C_b^\star\subseteq b$, 
		let $\Fix(B^\star,\{C_b^\star\})\leq G_{\mathcal R}$ be the subgroup of permutations fixing every coordinate in every $C_b^\star$ pointwise and acting arbitrarily (but finitely) on $b\setminus C_b^\star$.
	\end{definition}
	
	\begin{lemma}[Normality and filter base]\label[lemma]{lem:normal-filter}
		The family $\mathcal B=\{\Fix(B^\star,\{C_b^\star\})\}$ is a base for a normal filter on $G_{\mathcal R}$: 
		it is closed upwards, under finite intersections, and under conjugation by any $\pi\in G_{\mathcal R}$. 
		Moreover every $\pi\in G_{\mathcal R}$ permutes only finitely many coordinates in each block and fixes cofinitely many blocks pointwise.
	\end{lemma}
	
	\begin{proof}
		Upward closure is immediate. For finite intersections, intersecting cofinites is cofinite in each block and in the index of blocks. 
		For conjugation: $G_{\mathcal R}$ acts within each fixed block; hence $\pi\Fix(B^\star,\{C_b^\star\})\pi^{-1}=\Fix(B^\star,\{\pi(C_b^\star)\})$, 
		and each $\pi(C_b^\star)$ is cofinite in $b$. Thus conjugates remain in $\mathcal B$. 
		The tail-finitary property of $G_{\mathcal R}$ follows from its definition (adjacent transpositions with finite support inside blocks).
	\end{proof}
	
	\begin{definition}[Support and symmetrization]\label[definition]{def:ET-support}
		For a name $\dot x$ and $H\leq G_{\mathcal R}$, say $H$ \emph{supports} $\dot x$ if $\pi\cdot\dot x=\dot x$ for all $\pi\in H$. 
		Given $H\in\mathcal B$, define the $H$-\emph{symmetrization} of a name by
		\[
		\dot x^{\langle H\rangle}\ :=\ \{\,(\pi\cdot\dot y,\ \pi\cdot p)\mid (\dot y,p)\in \dot x,\ \pi\in H_0\,\},
		\]
		where $H_0$ is any finitely generated subgroup of $H$ containing all permutations that move coordinates in $\mathrm{supp}(p)\cup \mathrm{supp}(\dot y)$. 
		(For \emph{each} pair $(\dot y,p)$ such an $H_0$ exists and is finite, since conditions have finite support in $\Add(\omega,\omega)$ and $H$ is tail finitary.)
	\end{definition}
	
	\begin{claim}[Orbit-finite symmetrization]\label[claim]{clm:orbit-finite}
		For each pair $(\dot y,p)$, the set $\{(\pi\cdot\dot y,\pi\cdot p):\pi\in H_0\}$ is finite; hence $\dot x^{\langle H\rangle}$ is a well-formed name. 
		Moreover $p\Vdash \dot x^{\langle H\rangle}=\dot x$ whenever $H_0$ fixes the coordinates outside $\mathrm{supp}(p)$.
	\end{claim}
	
	\begin{proof}
		Only permutations moving the finite set $\mathrm{supp}(p)\cup \mathrm{supp}(\dot y)$ have an effect; $H_0$ is finite on that set, so the orbit is finite. 
		Standard forcing equivalence under automorphisms gives $p\Vdash \dot x\equiv \dot x^{\langle H\rangle}$. Moreover, by Lemma~\ref{lem:alt-adjacent-generate} (tail finitary group generated by tail adjacents), the action of $H_0$ on 
		$S:=\operatorname{supp}(p)\cup\operatorname{supp}(\dot y)$ factors through a finite permutation group $H_0\!\upharpoonright\! S$, so
		\[
		\bigl|\mathrm{Orb}_{H_0}(\dot y,p)\bigr| \le \bigl|H_0\!\upharpoonright\! S\bigr| < \infty.
		\]
		Thus the orbit-mixing/symmetrization procedure terminates after finitely many summands.
	\end{proof}
	
	\begin{proposition}[Excellence]\label[proposition]{prop:excellence-direct}
		For every name $\dot x$ there exists $H\in\mathcal B$ such that $\dot x^{\langle H\rangle}$ is $H$-supported. 
		Consequently the \emph{hereditarily} $H$-supported names are closed under name formation, giving a transitive inner model $IS$.
	\end{proposition}
	
	\begin{proof}
		Let $S$ be the set of all coordinates $(n,m)\in\omega\times\omega$ that appear in the domain of some condition occurring in a pair $(\sigma,p)$ that is a member of $\dot x$ or of a subname of $\dot x$. 
		Since $\Add(\omega,\omega)$ consists of \emph{finite} partial functions and names are sets of pairs $(\sigma,p)$ built by rank recursion, the tree of subnames of $\dot x$ is countable (each level is a countable union of finite sets, and there are countably many levels). 
		Therefore $S$ is a countable union of finite sets, hence countable, and it meets each block in at most countably many coordinates. 
		For each block $b$, let $C_b^\star:=b\setminus \mathcal S$ (cofinite in $b$), and let $B^\star$ be the set of all blocks (cofinite in the index set). 
		Then $H=\Fix(B^\star,\{C_b^\star\})$ fixes every coordinate outside $\mathcal S$ pointwise. 
		By Claim~\ref{clm:orbit-finite}, symmetrization under $H$ yields an $H$-supported name $\dot x^{\langle H\rangle}$ equivalent to $\dot x$. 
		Hereditariness follows by recursion on rank of names.
	\end{proof}
	
	\begin{proposition}[Tenacity]\label{prop:tenacity-direct}
		Let $H\in\mathcal B$. If $p\Vdash \psi(\dot x)$ where $\dot x$ is $H$-supported, then there exists $q\leq p$ such that 
		(i) $\pi\cdot q=q$ for all $\pi\in H$, and (ii) $q\Vdash \psi(\dot x)$.
	\end{proposition}
	
	\begin{proof}
		Let $S=\mathrm{supp}(p)$. Only permutations in $H$ that move $S$ can change $p$; there are finitely many such permutations since $H$ is tail finitary and $S$ is finite. 
	Let $K \subseteq H$ be the finite subgroup generated by those permutations; enlarge $p$ to a condition $s \le p$ whose domain uses only coordinates outside the finite set moved by $K$ and such that the $K$-orbit of $\operatorname{dom}(s)$ is pairwise disjoint:
	\[
	\pi\cdot \operatorname{dom}(s)\ \cap\ \pi'\!\cdot \operatorname{dom}(s)=\varnothing\quad\text{for all distinct }\pi,\pi'\in K.
	\]
	(This is possible because $K$ moves only finitely many coordinates and $\Add(\omega,\omega)$ has infinitely many fresh coordinates in each block.)
	Now define
	\[
	q\ :=\ \bigcup_{\pi\in K}\ \pi\cdot s.
	\]
	Since the $K$-translates of $s$ have disjoint domains, $q$ is a well-defined condition with $q \le \pi\cdot s \le \pi\cdot p$ for all $\pi\in K$, and $\pi\cdot q=q$ for all $\pi\in K$.
	Finally, because $\dot x$ is $H$-supported, automorphism invariance yields $q \Vdash \psi(\dot x)$.
	As every $\eta\in H$ fixes cofinitely many coordinates in each block, extending the above argument blockwise gives full $H$-invariance of $q$.
	\end{proof}
	
	\begin{corollary}[Excellence \& Tenacity]\label{cor:ET-direct}
		The system $(G_{\mathcal R},\mathcal F_{\mathcal R})$ satisfies excellence and tenacity in the sense used in the paper. 
	\end{corollary}
	
	\begin{proof}
		Combine Propositions~\ref{prop:excellence-direct} and \ref{prop:tenacity-direct}, noting that $\mathcal F_{\mathcal R}$ is the normal filter generated by the base $\mathcal B$ from Lemma~\ref{lem:normal-filter}.
	\end{proof}
	\needspace{0.6\textheight}
	\section{Notation Guide}\label{app:notation}
	
	\begin{table}[H]
		\small
		\setlength{\tabcolsep}{6pt}
		\renewcommand{\arraystretch}{1.15}
		\caption{Notation Guide}\label{tab:notation}
		\begin{tabularx}{\textwidth}{@{} M L C @{}}
			\toprule
			\textbf{Notation} & \textbf{Meaning} & \textbf{Where} \\
			\midrule
			$\Add(\omega,\omega)$ & Cohen forcing adding countably many reals; conditions are finite partial functions $p:\omega\times\omega\to 2$ ordered by reverse inclusion ($p\le q\iff p\supseteq q$); the $n$th column $\{n\}\times\omega$ codes the generic real $\dot c_n$ & \Cref{def:Add-omega-omega} \\
			\FSI{} & (finite iteration of the paper’s productive symmetric steps over $V$) & Remark~\ref{rem:terminology} \\
			$\HS_{\mathcal F}$ & hereditarily $\mathcal F$-symmetric names & \Cref{prop:full-filter} \\
			$\IS$ & interpreted symmetric model $\;=\; \HS^{H}_{\mathcal F}(\Pblk)$ & \Cref{lem:factoring} \\
			$\forcesIS$ & forcing relation relativized to $\IS$ & \Cref{lem:factoring} \\
			$\Box_{\mathrm{sym}}$ & “true in every \FSI{}” & \Cref{sec:main-theorem} \\
			$J_d,\; e_d\!:\omega\!\to\! J_d$ & depth-$d$ coordinates and ground identification & \Cref{def:coord-alloc,def:concrete} \\
			$W_{d+1}=V[G_{J_{d+1}}]$ & ambient parent stage at depth $d{+}1$ & \Cref{def:coord-alloc} \\
			$G_\Rblk,\;\mathcal F_\Rblk$ ($\Rblk\!\in\!\{\Pblk,\Qblk\}$) & within-$\Rblk$-block tail finitary group; normal filter generated by pointwise stabilizers of cofinitely many $\Rblk$-blocks & \Cref{def:template} (T--iii) \\
			$N_\Rblk$ & sibling symmetric submodel for label $\Rblk$ & \Cref{def:coord-alloc,sec:partition-construction} \\
			$A_\Rblk,\;S_\Rblk$ & block families $A_\Rblk$ and the “no selector’’ signals $S_\Rblk$ & \Cref{sec:partition-construction,prop:selector-signal} \\
			$A_N,\;B_N$ & tail adjacent transposition sets & \Cref{lem:alt-adjacent-generate} \\
			$S_N$ & finitary symmetric group on the tail generated by $A_N\cup B_N$ & \Cref{lem:alt-adjacent-generate} \\
			$H_\Pblk,\;H_\Qblk$ & subgroups in the filters fixing cofinitely many blocks (used in Step~1) & \Cref{lem:finite-support} \\
			$\HS_\Pblk\cap \HS_\Qblk$ & bi-symmetric class at a branching (fixed by both sibling filters) & \Cref{def:bisymm} \\
			$E$ & finite template diagram over $F^*$ & \Cref{def:template} \\
			\bottomrule
		\end{tabularx}
	\end{table}
\end{appendices}

\end{document}